\documentclass[12pt,a4paper]{article}
\usepackage[margin=25mm]{geometry}
\usepackage[utf8]{inputenc}

\usepackage{amsmath,amsfonts,amssymb, amsthm, datetime,relsize}
\usepackage{enumitem}
\usepackage{hyperref}
\usepackage{dsfont}

\numberwithin{equation}{section}
\numberwithin{figure}{section}

\makeatletter
\renewcommand*\env@cases[1][1.2]{%
  \let\@ifnextchar\new@ifnextchar
  \left\lbrace
  \def\arraystretch{#1}%
  \array{@{}c@{\quad}l@{}}%
}
\makeatother

\theoremstyle{plain}
\newtheorem{theorem}{Theorem}[section]
\newtheorem{lemma}[theorem]{Lemma}
\newtheorem{proposition}[theorem]{Proposition}
\newtheorem{corollary}[theorem]{Corollary}
\theoremstyle{definition}
\newtheorem{definition}[theorem]{Definition}
\newtheorem{remark}[theorem]{Remark}

\newcommand{\pp}{\mathfrak{p}} %pressure
\newcommand{\EE}{\calE}
\newcommand{\damp}{\calD_\alpha}
\newcommand{\BV}{\mathrm{BV}}
\newcommand{\dom}{\operatorname{dom}}
\newcommand{\dv}{\operatorname{div}}

\newcommand{\wt}{\widetilde} 
\newcommand{\ol}{\overline}

\def\dd{\,\mathrm{d}} % differential for integration 
\def\ee{\mathrm{e}}     % Euler's number 

\def\R{{\mathbb R}} \def\N{{\mathbb N}}  
 
\newcommand{\supp}{\operatorname{supp}}

\newcommand{\set}[2]{ \{\, #1 \: | \: #2 \,\} }

  \def\bbI{{\mathbb I}}

\def\bbV{{\mathbb V}}  
 
\def\calD{{\mathcal D}} \def\calE{{\mathcal E}} \def\calF{{\mathcal F}}

\begin{document}

\title{
Time-asymptotic self-similarity
of \\the
damped compressible Euler equations\\
in parabolic scaling variables
}

\author{Thomas Eiter\thanks{Freie Universit\"at Berlin, Department of Mathematics and Computer Science, Arnimallee 14, 14195 Berlin, Germany.} \textsuperscript{,}%
\thanks{Weierstrass Institute for Applied Analysis and Stochastics, Mohrenstr.~39, 10117 Berlin, Germany. \\
Email: thomas.eiter@wias-berlin.de\\
\phantom{Email: }stefanie.schindler@wias-berlin.de}
\ and Stefanie Schindler\footnotemark[2]}

\date{}
 
\maketitle

\begin{abstract}
We study the long-time behavior of solutions to the compressible Euler equations 
with frictional damping
in the whole space,
where we prescribe direction-dependent values for the density at spatial infinity.
To this end, we transform the system into parabolic scaling variables
and derive a relative entropy inequality,
which allows to conclude the convergence of the density
towards a self-similar solution to the porous medium equation
while the associated limit momentum 
is governed by Darcy's law.
Moreover, we obtain convergence rates that
explicitly depend on the flatness of the limit profile.
While we focus on weak solutions in the one-dimensional case,
we extend our results to energy-variational solutions in the multi-dimensional setting.

\medskip

\smallskip

\end{abstract}

\section{Introduction}

Consider the flow of 
an isentropic gas
subject to external friction, such as a flow through a porous medium,
described by the damped Euler equations. We first consider these equations
on the real line, given by
\begin{equation} \label{eq:Euler.original}
\begin{aligned}
\wt \rho_t + \wt m_x &= 0, \\
\wt m_t + \Big(\frac{\wt m^2}{\wt\rho} + \pp(\wt \rho) \Big)_x + \alpha \wt m
&=0,
\end{aligned}
\end{equation}
where $t\in(0,\infty)$ and $x\in\R$ denote time and space variables, respectively. 
The functions $\wt \rho = \wt \rho(t,x) \geq 0$ and $\wt m = \wt m(t,x)\in\R$ denote the density and the momentum of the fluid.
Moreover, $\pp = \pp(\wt \rho)$ represents the pressure as a function of the density, where $ \pp'(z) >0$ for all $z\in(0,\infty)$. 
An important particular case is the polytropic gas, 
which fulfills the so-called $\gamma$-law 
\begin{equation}\label{eq:pres.gamma}
\pp(\wt \rho) = k \wt \rho^\gamma \quad\text{for some } \gamma\geq 1, \ k>0.
\end{equation}
The term $\alpha \wt m$ in the momentum equation describes frictional damping
with friction coefficient $\alpha \geq 0$.
The initial data 
$(\wt\rho_0, \wt m_0)=(\wt \rho(0,\cdot),\wt m(0,\cdot))$ satisfy
\[
(\wt \rho_0,\wt m_0)(x) \to (\rho_\pm,0)
\quad \text{ as } x \to \pm \infty
\]
for given constants $\rho_+,\rho_-\geq 0$,
which corresponds to possibly infinite mass reservoirs
at positive and negative spatial infinity.

The natural entropy for the damped hyperbolic system~\eqref{eq:Euler.original}
is given by 
\[
\wt \eta(\wt\rho,\wt m)
= \frac{\wt m^2}{2\wt\rho}  +h(\wt \rho),
\]
which is the sum of kinetic and internal energy.
Here, the function $h$ and the pressure $\pp$ are linked to each other through
\begin{equation} \label{eq:h.properties}
\pp(z) = z  h'(z) - h(z)
\quad\text{ and }\quad
h''(z) = \frac{\pp'(z)}{z},
\end{equation}
where the second identity follows directly from the first one. 
If we assume~\eqref{eq:pres.gamma},
then we may choose the potential $h$ of the form
\begin{equation}
\label{eq:h_gammaLaw}
h(\wt \rho) = \begin{cases} \frac{1}{\gamma{-}1} \pp(\wt \rho) = \frac{k}{\gamma{-}1} \wt \rho^\gamma &
  \text{for } \gamma >1,\\ 
k \, \wt\rho \log \wt\rho& \text{for } \gamma = 1.
\end{cases}
\end{equation}
The entropy flux associated with $\wt \eta$ is 
given by
\[
\wt q(\wt \rho, \wt m) = \frac{\wt m}{\wt\rho}\big( \frac{\wt m^2}{2\wt\rho}   + \wt \rho  h'(\wt \rho)\big).
\]
Then, smooth solutions to \eqref{eq:Euler.original} satisfy the local entropy identity
\begin{equation} \label{eq:local.entr.original}
    \wt \eta(\wt \rho, \wt m)_t + \wt q (\wt \rho, \wt m)_x +\alpha \frac{\wt m^2}{\wt\rho}=0,
\end{equation}
which expresses the energy evolution during the process.

In this article, we are interested in the long-time behavior of solutions 
to the damped hyperbolic system~\eqref{eq:Euler.original}.
If the equations were formulated in a finite interval in space
with suitable boundary conditions,
one would expect solutions to converge to a steady state as $t\to\infty$.
In contrast, system~\eqref{eq:Euler.original}
is formulated on the real line,
and we may have an infinite mass supply at $\pm\infty$.
While for $\rho_+=\rho_-$ the convergence to a constant equilibrium state is expected,
the limit behavior in the case $\rho_+\neq\rho_-$
is more involved.
Hsiao and Liu~\cite{HsiLiu93NDP} firstly proved that the asymptotic behavior of solutions to~\eqref{eq:Euler.original} can be described by the self-similar solutions $(\wt \rho_{\mathrm{se}},\wt m_{\mathrm{se}})$
of the equations
\begin{equation*}
\begin{aligned}
\wt \rho_t + \wt m_x &= 0, \\
\pp(\wt \rho)_x + \alpha \wt m &= 0,
\end{aligned}
\end{equation*}
with the same spatial limits $(\wt \rho(x),\wt m(x))\to (\rho_\pm,0)$ as $x\to \pm \infty$.
Equivalently, this system can be written as the (parabolic) porous medium equation for the density $\wt \rho$,
where the momentum $\wt m$ is determined by Darcy's law,
that is,
\begin{equation} \label{eq:PME.Darcy.original}
\wt \rho_t = \frac1\alpha \pp(\wt \rho)_{xx}, 
\qquad
\wt m = -\frac1{\alpha} \,\pp(\wt \rho)_x.  
\end{equation}
In the literature, such self-similar solutions $(\wt\rho_{\mathrm{se}}, \wt m_{\mathrm{se}})$ are usually called \textit{nonlinear diffusion waves}, as for example in \cite{CuYiZh18CND}.
In the work \cite{HsiLiu93NDP}, 
a uniform convergence rate was shown 
under increased regularity assumptions on the solution $(\rho,m)$.
In~\cite{HsiLiu92CNDW}, the same authors derived similar results
for the damped $p$-system,
which is a formulation of system~\eqref{eq:Euler.original} in Lagrangian coordinates.
Since then, these results have been extended in various ways, see \cite{Nish96CRNDW, SerXia97ABLWES, NiWaYa00LpCNDW, HuPa03CRCE, Mei09NDW, GenTan19CRND} and references therein, especially by derivation of improved convergence rates but also by allowing for more general solution classes. 
The equations have also been studied subject to a time-dependent damping, see for instance \cite{CuYiZh18CND, GeHuWu22L2CNDW, GeHuaJuWu23TAE}.
In the case of finite mass, we refer to \cite{HuFeMa05CBS, HuPaWa11L1CBS, GenHua19L1CBS}, where the convergence towards the famous Barenblatt solution
is shown.

In the paper \cite{LatTza2013REDR},
the so-called diffusive scaling of the system~\eqref{eq:Euler.original} is considered,
and it is shown that weak entropy solutions converge to
smooth solutions $(\wt\rho,\wt m)$ to~\eqref{eq:PME.Darcy.original} 
as the scaling parameter goes to zero.
This convergence is quantified in terms of the relative entropy.
Here, we follow this idea of deriving relative entropy estimates, 
but instead of the physical variables $(t,x)$, we perform the analysis
in parabolic scaling variables defined by
\begin{equation}
\label{eq:scaling.variables}
\tau := \log(1{+}t) \quad \text{ and } \quad y:= \frac{x}{\sqrt{1{+}t}}.
\end{equation}
We transform the equations \eqref{eq:Euler.original} into these new coordinates.
With $\rho(\tau,y) := \wt \rho(t,x)$ and $n(\tau,y) := \ee^{\tau/2} m(\tau,y) = (1{+}t)^{1/2}\, \wt m(t,x)$,
a straightforward computation yields 
\begin{equation} \label{eq:Euler.trans}
\begin{aligned}
\rho_\tau &= \frac{y}{2} \rho_y - n_y, \\
n_\tau &= \frac{y}{2} n_y + \frac12 n - \Big( \frac{n^2}{\rho} \Big)_y - \ee^\tau \big( \pp(\rho)_y + \alpha n\big),
\end{aligned}
\end{equation}
while the according initial data $(\rho_0,n_0)= (\rho(0,\cdot),n(0,\cdot))$ satisfy the asymptotic conditions
\begin{equation} \label{eq:BC}
    (\rho_0,n_0)(y) \to (\rho_\pm,0)
    \quad \text{ as }y\to \pm \infty .
\end{equation}
Observe that we scale the magnitude of the momentum variable 
to ensure that Darcy's law can be fulfilled asymptotically in time. 
Indeed, 
by this scaling, the pressure term and the damping term in~\eqref{eq:Euler.trans}
have the same exponentially growing prefactor,
which indicates that Darcy's law needs to be satisfied asymptotically to equilibrate the system.
Moreover, the scaling creates the new drift terms $\frac{y}{2} \rho_y $ and $\frac{y}{2} n_y$, which will provide the important terms in the relative entropy inequality
that lead to exponential decay
with respect to the scaled time $\tau$.

Transformation of
the porous medium equation from~\eqref{eq:PME.Darcy.original} 
into parabolic scaling variables and combination with the spatial limits yields
 \begin{equation} \label{eq:PME.trans}
     \rho_\tau - \frac y2 \rho_y = \frac1\alpha \pp(\rho)_{yy} \quad \text{ with } \quad \lim_{ y \to \pm \infty} \rho(\tau,y) = \rho_\pm.
 \end{equation}
Steady-state solutions $\rho^*=\rho^*(y)$ of~\eqref{eq:PME.trans} correspond to self-similar solutions $\wt \rho_{\mathrm{se}}$ to the original porous medium equation via $\wt \rho_{\mathrm{se}}(t,x) = \rho^*(y)$,
and these self-similar solutions appear as long-time limits of general solutions $\wt\rho$.
Expressed in scaling variables $(\tau,y)$, 
this means that the solution $\rho$ to~\eqref{eq:PME.trans} 
converges to $\rho^*$ as $\tau\to\infty$.
In this paper,
we perform a similar convergence study 
for solutions $(\rho,n)$ to the scaled damped Euler equations \eqref{eq:Euler.trans}, 
which for $\tau \to \infty$ converge towards the similarity profile $(\rho^*,n^*)$ 
characterized by the profile equations
\begin{equation} \label{eq:profileEq}
\begin{aligned}
0 = \frac1\alpha \pp(\rho^*)_{yy} + \frac y2 \rho^*_y, \quad \
0 = n^* + \frac1\alpha \pp(\rho^*)_y
\qquad \text{with}\ 
\lim_{y\to\pm\infty}(\rho^*,n^*)(y) =(\rho_\pm, 0).
\end{aligned}
\end{equation}
This approach is in accordance with~\cite{MieSch24CSSP},
where the long-time behavior of solutions to a reaction-diffusion system was studied.
The scaling~\eqref{eq:scaling.variables} 
emphasizes the time-asymptotic parabolic behavior
of the system,
and it was already used by Gallay and Raugel~\cite{GalRau98SVAE,GalRau97STWD,GalRau00SVSH} to analyze the asymptotic behavior of scalar damped hyperbolic equations. 
To the authors' knowledge, the present paper is the first one to adapt this method to a damped hyperbolic system
and to combine it with a relative entropy approach.

The general idea of proving convergence by a relative entropy estimate
traces back to the works by Dafermos \cite{Dafe79SMTF, Dafe79SLTS} and DiPerna \cite{Dipe79USHC}.
It has been refined and generalized by many authors in recent years; see for example \cite{HuaPan06ABSDCE, LatTza2013REDR, GenHua19L1CBS, GeHuWu22L2CNDW} for problems related the damped Euler equations.
We shall first show that the relative entropy $\EE$
satisfies an inequality
of the form
\begin{equation}
\label{eq:relent.intro}
\frac{\mathrm d}{\mathrm d \tau}\EE(\tau)
+ \damp(\tau)+\frac{1}{2}\EE(\tau)\leq\Xi,
\end{equation}
where the term $\frac12  \EE(\tau )$ arises from the drift terms 
$\frac{y}{2} \rho_y $ and $\frac{y}{2} n_y$ in~\eqref{eq:Euler.trans}
and is thus due to the parabolic scaling. 
Moreover, $\damp\geq 0$ is a dissipative term associated with the frictional damping, while the term $\Xi$ does not have a clear sign
and has to be controlled by the other terms.
To this end, the term $\frac12  \EE(\tau )$
is necessary,
and its appearance can be considered the main effect of the transformation into scaling variables.
Moreover, it finally leads to the exponential decay of $\EE(\tau)$ as $\tau\to\infty$
by means of Gronwall's inequality. 
However, since the limiting profile is not a true steady state if we consider different limits at spatial infinity,
the function $\EE$ cannot be a true Lyapunov function.
Instead, we only show that $\EE$ becomes a Lyapunov function for large $\tau$, which is sufficient for the convergence and comparable with the results in \cite{MieSch24CSSP}.

Notice that smooth solutions to~\eqref{eq:Euler.trans} even satisfy~\eqref{eq:relent.intro} as an equality.
Since those exist only up to a finite time in general,
we consider a framework of weak solutions instead,
where results on the global existence 
of solutions satisfying an entropy inequality
are available in the one-dimensional case,
see~\cite{DiCheLuo89Conv,HuaPan06ABSDCE}.
However, in the multi-dimensional case,
analogous results for weak solutions seem to be out of reach,
which is why 
we instead study
a further relaxed solvability concept in this setting, 
namely energy-variational solutions.
This concept was recently introduced for a general class of hyperbolic conservation laws in~\cite{EitLas24envarhyp}, 
including the compressible Euler equations,
and has further been investigated for other models, see~\cite{ALR24envarvisc,EiHoLa23wsevsviscoel,LasRei23envarEL,Las24envarMDinc,Las25envarstr}.
We establish the global existence of energy-variational solutions
and show that our results on
the long-time behavior of weak solutions 
can be transferred to this case.

The article is structured as follows. 
In Section~\ref{se:MainResults} we introduce 
weak solutions 
to the scaled system~\eqref{eq:Euler.trans},
and we present the main results on the convergence towards 
solutions of the profile equations~\eqref{eq:profileEq} as $\tau\to\infty$
in the one-dimensional case.
In Section~\ref{sec:relent} 
we derive a relative entropy inequality
for weak solutions,
which we employ in Section~\ref{se:ConvProof} 
to prove the main results in the one-dimensional case. 
Finally, in Section~\ref{sec:envar} we 
consider energy-variational solutions to the damped Euler equations in multiple dimensions. 
We show their global-in-time existence,
and we extended our results on the long-time behavior
to this generalized solution class
in a straightforward way.

\section{Convergence results in one dimension} \label{se:MainResults}

This section is dedicated to presenting our main results concerning the convergence of weak solutions to the scaled system~\eqref{eq:Euler.trans} towards the similarity profile.
We consider weak solutions in the following sense.

\begin{definition} \label{def:weakEntSol.trans}
    Let $(\rho_0,n_0)\in L^1_\mathrm{loc}(\R)^2$ such that 
    $\rho_0\geq0$ a.e.
    We call a pair $(\rho, n)\in L^1_{\mathrm{loc}}([0,\infty)\times \R)^2$ 
     with $\frac{n^2}{\rho},\,\pp(\rho)\in L^1_{\mathrm{loc}}([0,\infty)\times \R)$ 
    a \textit{weak solution} to \eqref{eq:Euler.trans} if it satisfies $\rho\geq 0$ a.e.~and 
    \begin{align} 
    \label{eq:conteq.weak}
       - \int_0^\infty\!\!\!\int_{\R} \phi_{\tau} \,  \rho + \phi_y \, \big(- \frac y2 \rho + n \big) - \frac12 \phi \rho  \dd y \dd \tau & = \int_\R \phi(0,y) \, \rho_0  \dd y, \\
       \label{eq:momeq.weak}
        - \int_0^\infty\!\!\!\int_{\R} \psi_{\tau} \, n + \psi_y \, \Big(- \frac y2 n + \frac{ n^2}{\rho} + \ee^{\tau} \pp(\rho) \Big) -  \psi \, \ee^{\tau} \alpha n  \, \dd y \dd \tau  
      & =  \int_\R \psi(0,y) \, n_0 \dd y ,
    \end{align}
   for all $\phi,\psi \in C_c^\infty([0,\infty) \times \R)$.
\end{definition}
 
Notice that one can also derive this weak formulation by transforming the weak formulations of the unscaled equations \eqref{eq:Euler.original} 
into parabolic coordinates.

Here, the quotient $\frac{n}{\rho}$
is only well defined either if $\rho\neq 0$, or if $\rho=0$ and $n=0$,
where we set $\frac{0}{0}:= 0$. 
In particular, 
the assumption
$\frac{n^2}{\rho}\in L^1_{\mathrm{loc}}([0,\infty)\times \R)$
implies that $n=0$ almost everywhere where $\rho=0$.

Additionally, the weak solutions need to satisfy a suitable form of an entropy inequality.
To begin with, we write the entropy pair in terms of the new variables $(\tau,y)$ and for the functions $(\rho,n)$, so that it is compatible with the transformed system~\eqref{eq:Euler.trans}.
We define the entropy function and the associated flux by 
\[
\eta(\tau; \rho,n) := \ee^{-\tau} \frac{n^2}{2 \rho} + h(\rho) \quad \text{ and } \quad q(\tau; \rho,n) := \ee^{-\tau}  \frac{n^3}{2\rho^2} + n h'(\rho),
\]
where the exponential factors result from the scaling of $m=\ee^{-\tau/2} n$,
and we additionally scaled the magnitude of $q$ 
in such a way that~\eqref{eq:local.entr.original} 
is transformed into
    \begin{equation} \label{eq:eta.diss.trans}
    \eta(\tau;\rho, n)_\tau - \frac{y}{2}\eta(\tau;\rho, n)_y + q (\tau;\rho, n)_y +\alpha \frac{n^2}{\rho} = 0.
      \end{equation}
For $\ol\rho\geq 0$ and $\ol n\in\R$,
we further introduce the corresponding relative entropy and relative entropy flux by
\begin{equation}
\label{eq:relentr.relentflux}
\begin{aligned}
\eta(\tau;\rho,n\,\vert\,\ol\rho,\ol n)
&=\frac{1}{2}\ee^{-\tau}\rho\Big|\frac{n}{\rho}-\frac{\ol n}{\ol \rho}\Big|^2 + h(\rho\,\vert\,\ol\rho),
\\
q(\tau;\rho,n\,\vert\,\ol\rho,\ol n)
&=\frac{1}{2}\ee^{-\tau}n\Big|\frac{n}{\rho}-\frac{\ol n}{\ol \rho}\Big|^2 
+ \rho \big( h'(\rho)-h'(\ol\rho) \big)\Big(\frac{n}{\rho}-\frac{\ol n}{\ol\rho}\Big) 
+ \frac{\ol n}{\ol\rho} h(\rho\,\vert\,\ol\rho),
\end{aligned}
\end{equation}
respectively,
where the relative internal energy is defined by
\[
h(\rho\,\vert\,\ol\rho)=h(\rho)-h(\ol\rho)-h'(\ol\rho)(\rho-\ol\rho).
\] 
If $(\rho,n)$, $(\ol\rho,\ol n)$ 
are suitable functions on $\R$,
we define the total relative entropy by
\begin{equation}
\label{eq:relent.total}
\EE(\tau;\rho,n\,\vert \,\ol\rho,\ol n)
=\int_\R \eta\big(\tau;\rho(y),n(y)\,\vert\,\ol\rho(y),\ol n(y)\big)\dd y.
\end{equation}
Notice that $\EE(\tau;\rho,n\,\vert \,\ol\rho,\ol n)$ is only finite
if the reference function $(\ol\rho,\ol n)$ has the same behavior at infinity as $(\rho,n)$. 
Moreover, by finiteness of $\EE(\tau;\rho,n\,\vert \,\ol\rho,\ol n)$ we implicitly assume that
$\rho,\ol\rho\geq0$ a.e.
Special reference densities are given by the function
\begin{equation}
    \label{eq:rho.hat}
    \hat\rho\colon\R\to\R,\qquad
    \hat\rho(y)=
    \begin{cases}
        \rho_- & \text{for } y<0,
        \\
        \rho_+ & \text{for }y\geq 0,
    \end{cases}
\end{equation}
and by its smoothened version $\rho_\infty \in C^1(\R)$ with
\begin{equation}\label{eq:rhoinfty}
    \exists \, R>0: \qquad \rho_\infty(y)=\rho_+ \text{ for }y>R,
    \qquad
    \rho_\infty(y)=\rho_- \text{ for }y<-R.
\end{equation}
We study the asymptotic conditions~\eqref{eq:BC} in the sense that
$\EE(0;\rho_0,n_0\,\vert \,\hat\rho,0)<\infty$,
which is equivalent to $\EE(0;\rho,n\,\vert \,\rho_\infty,0)<\infty$,
as follows from Proposition~\ref{prop:relent.difference} below.

The following two versions of entropy inequalities can now be formulated.
\begin{enumerate}[label=\textbf{(E\arabic*)}]
    \item
    \label{item:E1}
    It holds $\rho, \frac{n}{\rho}\in L^\infty_\mathrm{loc}([0,\infty);L^\infty(\R))$
     and $h(\rho)\in L^1_{\mathrm{loc}}([0,\infty)\times \R)$, 
    and for all $\chi\in C_c^\infty([0,\infty) \times \R)$ with $\chi\geq 0$, we have
       \begin{equation}
   \label{eq:weakEnt.scaled}
   \begin{aligned}
   \int_0^\infty\!\!\!\int_\R {-}  \chi_{\tau} \eta (\tau; \rho,n)
   {+} \chi_y\Big(\frac{y}{2}\eta (\tau; \rho,n) {-} q(\tau;\rho,\eta)\Big)
   &{+} \chi\Big(\frac{1}{2}\eta (\tau; \rho,n) {+} \alpha \frac{n^2}{\rho}\Big)
   \dd y \dd\tau
   \\
   &
   \leq \int_\R \chi(0,y)\,\eta(0;\rho_0,n_0)\dd y.
   \end{aligned}  
   \end{equation}
   \item
 \label{item:E2}
  There exists $\rho_\infty \in C^1(\R)$ satisfying~\eqref{eq:rhoinfty}
 such that
 for all $\varphi \in C_c^\infty([0,\infty))$ with $\varphi\geq 0$, we have
 \begin{equation}
   \label{eq:weakEnt.total.rhoinf}
   \begin{aligned}
   &\int_0^\infty\! - \varphi_{\tau}\, \EE (\tau; \rho,n\,\vert\,\rho_\infty,0)
   + \varphi \Big(\frac{1}{2}\EE (\tau; \rho,n\,\vert\,\rho_\infty,0) + \int_\R\alpha \frac{n^2}{\rho}\dd y\Big) \dd\tau
   \\
   &\qquad\quad
   \leq  \varphi(0)\,\EE(0;\rho_0,n_0\,\vert\,\rho_\infty,0)
   + \int_0^\infty \!\varphi\int_\R h'(\rho_\infty)_y \Big(\frac{y}{2}(\rho-\rho_\infty)-n\Big)\dd y\dd\tau.
   \end{aligned}  
\end{equation}
\end{enumerate}
Inequality~\eqref{eq:weakEnt.scaled} is a weak formulation of~\eqref{eq:eta.diss.trans} relaxed to an inequality. It can equivalently be derived by transforming an analogously relaxed version of \eqref{eq:local.entr.original} into parabolic scaling variables $(\tau,y)$. 
The local entropy inequality~\eqref{eq:weakEnt.scaled} is a classical assumption,
and weak solutions satisfying~\eqref{eq:weakEnt.scaled} are usually called \textit{weak entropy solutions}.
Moreover, the additional boundedness assumptions on $\rho$ and $\frac{n}{\rho}$ in~\ref{item:E1}
seem to be necessary for the derivation of an inequality for the total relative entropy.
Observe that the global existence of weak solutions satisfying~\ref{item:E1}
was shown in~\cite{DiCheLuo89Conv} for $\pp(\rho)=k\rho^\gamma$ with $1<\gamma<3/2$,
which was extended in~\cite{HuaPan06ABSDCE} to
the case $\gamma <3$.

In contrast, this boundedness condition is not necessary if we assume~\ref{item:E2},
where the total relative entropy with respect to a given
reference density $\rho_\infty$
is finite.
During our analysis,~\eqref{eq:weakEnt.total.rhoinf} 
will turn out
to be a special case of a more general relative entropy inequality
derived in Theorem~\ref{thm:relentr.weak.E1}
under the condition~\ref{item:E1},
and it is thus less restrictive,
see also Remark~\ref{rem:E1.implies.E2} below. 
Moreover, the relative entropy inequality derived in Theorem~\ref{thm:relentr.weak.E2} shows that 
the criterion is independent of the specific choice of the reference density $\rho_\infty$,
and one may even consider a $\tau$-dependent reference density,
which would result in an additional term on the right-hand side of~\eqref{eq:weakEnt.total.rhoinf}.
We further have that inequality~\eqref{eq:weakEnt.total.rhoinf} is equivalent to the pointwise inequality
\[
\begin{aligned}
    &\EE (\tau; \rho(\tau),n(\tau)\,\vert\,\rho_\infty,0)
   + \int_\sigma^\tau\!\!\Big(\frac{1}{2}\EE (\tau'; \rho,n\,\vert\,\rho_\infty,0) + \int_\R\alpha \frac{n^2}{\rho}\dd y\Big) \dd\tau'
   \\
   &\qquad
   \leq \EE(\sigma;\rho(\sigma),n(\sigma)\,\vert\,\rho_\infty,0)
   + \int_\sigma^\tau\!\!\!\int_\R h'(\rho_\infty)_y \Big(\frac{y}{2}(\rho-\rho_\infty)-n\Big)\dd y\dd\tau'
\end{aligned}
\]
for a.a.~$\sigma,\tau\in[0,\infty)$ with $\sigma<\tau$, including $\sigma=0$ with $(\rho(0),n(0))=(\rho_0,n_0)$,
see~\cite[Lemma~2.5]{EitLas24envarhyp} for instance.

To prepare our main results in the one-dimensional case, 
we recall existence and decay properties of the 
similarity profiles for the porous medium equation in~\eqref{eq:PME.Darcy.original}
with non-zero limits.

\begin{theorem}
    \label{thm:profile}
Let the pressure $\pp$ be of the form~\eqref{eq:pres.gamma} and $\alpha>0$.
For all pairs $\rho_\pm >0$, there exists a function $\rho^*\in C^\infty(\R)$
that satisfies $0 = \frac1\alpha \pp(\rho^*)_{yy} + \frac y2 \rho^*_y$ in $\R$ and
$\rho^*(y)\to\rho_\pm$ as $y\to\pm\infty$. 
Moreover, $\rho^*$ is monotone and there are $C,c>0$ such that
\begin{equation}
\label{eq:pme.convergence}
\big|\rho^*(y)-\hat\rho(y)\big|
+|\rho^*_y(y)|+|\rho^*_{yy}(y)|
\leq C |\rho_+-\rho_-| \ee^{-c\alpha y^2}
\end{equation}
for all $y\in\R$, where $\hat\rho$ is defined in~\eqref{eq:rho.hat}.
\end{theorem}

\begin{proof}
    See \cite{VanPel77CSSN,HsiLiu92CNDW}.
    A modern approach, which also allows to handle similarity profiles 
    in the vector-valued case
    and uses the theory of monotone operators, 
    was established recently in \cite{MieSch25ESPS}.
\end{proof}

As a consequence, it holds
$0 < \min\{\rho_-,\rho_+\} \leq \rho^*(y) \leq \max\{\rho_-,\rho_+\}$ for all $y\in \R$.
Moreover, $\rho^*$, $\rho^*_y$ and $\rho^*_{yy}$ converge faster than exponentially as $y\to\pm\infty$.

Now, we can formulate our main results on the long-time behavior of solutions to the damped Euler equations, which will be proven in Section~\ref{se:ConvProof}.

\begin{theorem} \label{thm:Conv}
Let the pressure $\pp$ be of the form~\eqref{eq:pres.gamma}.
Let $\alpha>0$ and $\rho_\pm>0$  with $\rho_+\neq\rho_-$,
and let $\hat\rho$ be as in~\eqref{eq:rho.hat}.
Consider the unique profile $(\rho^*,n^*)$ defined by~\eqref{eq:profileEq}.
For initial data $(\rho_0,n_0)\in L^1_{\mathrm{loc}}(\R)^2$
such that $\rho_0\geq 0$
and $\EE(0;\rho_0,n_0\,\vert\,\hat\rho,0)<\infty$, 
let $(\rho,n)$ be a weak solution to~\eqref{eq:Euler.trans} 
satisfying~\ref{item:E1} or~\ref{item:E2}.
Then $\EE(0;\rho_0,n_0\,\vert\,\rho^*,n^*)<\infty$
and 
\begin{equation} \label{eq:convergence}
     \EE(\tau;\rho(\tau),n(\tau) \, \vert \, \rho^*, n^*) \leq \ee^{-(\frac12{-}\theta)\tau + \frac12 \mu} \Big( \EE(0;\rho_0,n_0 \, \vert \, \rho^*,n^*) + \frac{K}{\theta} \Big) \quad \text{ for a.a.~} \tau >0.
\end{equation}
Here, the constants $\theta$, $\mu$ and $K$ are given by 
\begin{align*}
  \theta := \max\{2,\gamma{-}1\}\big\Vert \big[ \tfrac{1}{\alpha}h'(\rho^*)_{yy} \big]_+ \big\Vert_{L^\infty(\R)}, 
  \; \ 
    \mu := \Big\| \frac{|R^*|}{2 k (\rho^*)^\gamma} + \frac{3|R^*|}{2\rho^*}  \Big\|_{L^\infty(\R)},
    \; \
    K := \int_\R |R^*|\dd y,
\end{align*}
where $[\cdot]_+:=\max\{0,\cdot\}$ and
\begin{equation}
\label{eq:R.def}
R^*:=  - \frac{y}{2} n^*_y - \frac12 n^* + \Big( \frac{ (n^*)^2}{\rho^*}\Big)_y.
\end{equation}
In particular, if $\theta<\frac{1}{2}$, then
$
\EE(\tau;\rho(\tau),n(\tau) \, \vert \, \rho^*, n^*)
\to 0
$
as $\tau\to\infty$.
Moreover, in this case, it holds
\begin{equation}\label{eq:convergence.dissipation}
    \alpha\int_\tau^\infty\!\!\! \int_\R \rho \,\Big\vert \frac n\rho - \frac{n^*}{\rho^*} \Big\vert^2 \dd y\dd \tau'
    \leq \ee^{-(\frac12{-}\theta)\tau + \frac12 \mu} \Big( \EE(0;\rho_0,n_0 \, \vert \, \rho^*,n^*) + \frac{K}{\theta} \Big) 
    +2K\ee^{-\tau/2} 
\end{equation}
for a.a.~$\tau\geq 2\log(\frac{2\mu}{1-2\theta})$.
\end{theorem}

We emphasize that all constants are explicit and do not depend on the solution $(\rho,n)$
but only on the limit vales $\rho_\pm$, the pressure $\pp$ and the damping coefficient $\alpha$.
Since $\rho^*(y)=\rho_*(\sqrt{\alpha}y)$ and $n^*(y)=\frac{1}{\sqrt{\alpha}}n_*(\sqrt{\alpha}y)$
if $(\rho_*,n_*)$ is the solution to~\eqref{eq:profileEq} for $\alpha=1$,
we see that
$\mu$ and $K$ scale as $\mu\sim\frac{1}{\sqrt{\alpha}}$ and $K\sim\frac{1}{\alpha}$,
while $\theta$ is independent of $\alpha$. 
In particular, independently of $\alpha$, 
the condition $\theta<\frac{1}{2}$ can be satisfied
by choosing $|\rho_+-\rho_-|$ small enough,
due to~\eqref{eq:pme.convergence}.
Such flatness conditions for the profile function also appear in other articles, see~\cite{HsiLiu92CNDW,Nish96CRNDW,MieSch24CSSP}
for instance.
One can also generalize Theorem~\ref{thm:Conv} to more general pressure laws that grow like~\eqref{eq:pres.gamma} asymptotically.
But then the occurring constants would be less explicit.

While the assumption $\EE(0;\rho_0,n_0\,\vert\,\hat\rho,0)<\infty$
is natural for the initial data, 
we can also assume $\EE(0;\rho_0,n_0\,\vert\,\ol\rho,\ol n)<\infty$
for any other pair $(\ol\rho,\ol n)$ such that $\ol\rho-\hat\rho$ 
and $\ol n$ have sufficient decay as $|y|\to\infty$,
see Proposition~\ref{prop:relent.difference} below.
For instance, the choice $(\ol\rho,\ol n)=(\rho_\infty,0)$ with $\rho_\infty$ satisfying~\eqref{eq:rhoinfty},
or the asymptotic profile $(\ol\rho,\ol n)=(\rho^*,n^\ast)$ are possible. 
In particular, we do not impose any smallness conditions on the initial data,
and even vacuum is allowed to occur, at least locally,
for the initial data as well as during the evolution.

Focusing on the kinetic energy, 
from~\eqref{eq:convergence} we obtain the pointwise bound 
\[
\ee^{-\tau}\int_\R \rho(\tau) \,\Big\vert \frac {n(\tau)}{\rho(\tau)} - \frac{n^*}{\rho^*} \Big\vert^2 \dd y \leq C \ee^{-(\frac12{-}\theta)\tau}
\]
for some $C>0$ and a.a.~$\tau>0$,
which does not imply any convergence of $n$ towards $n^*$ as $\tau\to\infty$.
Although~\eqref{eq:convergence.dissipation} does not give a pointwise bound either, 
it yields convergence in a sense of integral means. 

In the case of coincident limits $\rho_+=\rho_-$,
we derive a faster convergence rate than in~Theorem~\ref{thm:Conv}.
Moreover, the damping term is not necessary,
and we can consider more general pressure functions $\pp$.
We can even treat the case $\rho_\pm=0$ if we additionally assume that the pressure satisfies
\begin{equation}
\label{eq:pres.better}
\exists\, c>0: \ \pp'(z)\leq c \,\frac{\pp(z)}{z} 
\quad \text{and} \quad
\int_0^1\frac{\pp(z)}{z^2}\dd z<\infty.
\end{equation} 
Observe that~\eqref{eq:pres.better} is satisfied if $\pp$ is given by~\eqref{eq:pres.gamma} for $\gamma>1$, but not for $\gamma=1$.

\begin{theorem}\label{thm:convergence.samelimits}
Let $\alpha\geq0$ and $\rho_+=\rho_-\equiv\hat\rho> 0$, 
and let $\pp\in C([0,\infty))\cap C^1(0,\infty)$ with $\pp'>0$ in $(0,\infty)$.
Consider initial data $(\rho_0,n_0)\in L^1_{\mathrm{loc}}(\R)^2$ 
such that $\rho_0\geq 0$
and $\EE(0;\rho_0,n_0\,\vert\,\hat\rho,0)<\infty$.
Let $(\rho,n)$ be a weak solution to~\eqref{eq:Euler.trans} 
satisfying~\ref{item:E1} or~\ref{item:E2}.
Then 
\begin{equation}\label{eq:convergence.samelimits.entropy}
    \EE (\tau;\rho(\tau),n(\tau) \, \vert \, \hat \rho,0) \leq \ee^{-\frac{\tau}{2}} \; \EE(0;\rho_0,n_0 \, \vert \, \hat \rho, 0)  \quad \text{ for a.a.~}\tau >0
\end{equation}
and
\begin{equation}\label{eq:convergence.samelimits.dissipation}
    \alpha\int_\tau^\infty\!\!\! \int_\R \frac {|n|^2}{\rho} \dd y\dd \tau'
   \leq \ee^{-\frac{\tau}{2}} \; \EE(0;\rho_0,n_0 \, \vert \, \hat \rho, 0)
   \quad \text{ for a.a.~}\tau >0.
\end{equation}
If $\pp$ additionally satisfies \eqref{eq:pres.better},
then
the same statement holds also for $\rho_+=\rho_-\equiv\hat\rho= 0$.
\end{theorem}

In the case of coincident limits for $y\to\pm\infty$, where the self-similar profile is constant, 
the friction coefficient $\alpha$ can be set to zero
without affecting the convergence rate of the total relative entropy.
Only the faster convergence rate of the relative kinetic energy, 
coming from~\eqref{eq:convergence.samelimits.dissipation}, would be lost for $\alpha=0$.
In contrast, for different limits, we need $\alpha >0$ to define the self-similar profile by~\eqref{eq:profileEq}. Nevertheless, even in that case, the damping term can be omitted during the estimates without affecting the convergence 
rate of the relative entropy,
as will become evident during the proof.

Theorem~\ref{thm:convergence.samelimits} also covers the important case of
finite initial total mass. 
However, due to the parabolic scaling,
the mass corresponding to solutions to~\eqref{eq:Euler.trans} decays exponentially.
In particular, we obtain the zero profile $(\rho^*,n^*)\equiv(0,0)$ 
as limit objects in this case,
rather than the famous Barenblatt profile.

\begin{remark}\label{rem:convergence} 
Let us return to the original momentum $ m= \ee^{-\tau/2} n$
 and $m^*=\ee^{-\tau/2} n^*$.
Then Theorem~\ref{thm:Conv} (with $\theta <\frac12$) and Theorem~\ref{thm:convergence.samelimits} yield exponential 
convergence of the relative kinetic energy and the relative internal energy in the sense
\begin{equation*}
 \int_\R  \frac12 \rho(\tau) \, \Big| \frac{m(\tau)}{\rho(\tau)} - \frac{m^*}{\rho^*} \Big|^2 \dd y 
  + \int_\R h(\rho(\tau) \, \vert \, \rho^* ) \dd y 
   + \alpha\int_\tau^\infty\!\!\! \int_\R \ee^{\tau'}\!\rho \,\Big\vert \frac m\rho - \frac{m^*}{\rho^*} \Big\vert^2 \dd y\dd \tau'
  \leq C \ee^{-(\frac12 -\theta) \tau} 
\end{equation*}
for some $C>0$ and a.a.~$\tau>0$.
In terms of the original time variable $t = \ee^\tau-1$, this yields
the algebraic pointwise decay 
\begin{equation*}
 \int_\R \wt \eta \big(\wt \rho(t, y \sqrt{1{+}t}),\wt m (t, y \sqrt{1{+}t}) \, | \, \wt \rho_{\mathrm{se}}(t, y \sqrt{1{+}t}), \wt m_{\mathrm{se}} (t, y \sqrt{1{+}t}) \big) \dd y \leq C (1{+} t)^{-(\frac12 -\theta)},
\end{equation*}
where $\wt \rho_{\mathrm{se}}(t,x) = \rho^*(y)$ and $\wt m_{\mathrm{se}}(t,x) = m^*(y)$ denote the corresponding self-similar solution to~\eqref{eq:PME.Darcy.original}. 
Due to the strict convexity of the relative entropy $\wt \eta$, this results in the following pointwise convergences for a.a.~$y \in \R$:
\begin{equation*}
  \wt \rho(t, y \sqrt{1{+}t}) - \wt \rho_{\mathrm{se}}(t, y \sqrt{1{+}t}) \to 0 \quad \text{ and } \quad   \wt m(t, y \sqrt{1{+}t}) - \wt m_{\mathrm{se}}(t, y \sqrt{1{+}t}) \quad \text{ as } t \to \infty.
\end{equation*}
Note that in terms of the original spatial variable $x$, no corresponding statement can be made, as the transformation of the integral results in an additional factor $\ee^{-\tau/2}$.
We further obtain the estimate
\[
\alpha\int_t^\infty\!\!\! \int_\R \wt\rho(t,y\sqrt{1{+}t'}) \,\bigg\vert \frac {\wt m(t',y\sqrt{1{+}t'})}{\wt\rho(t',y\sqrt{1{+}t'})} - \frac{\wt m_{\mathrm{se}} (t', y \sqrt{1{+}t'})}{\wt \rho_{\mathrm{se}}(t', y \sqrt{1{+}t'})} \bigg\vert^2 \dd y\dd t' 
  \leq C (1{+} t)^{-(\frac12 -\theta)},
\]
which provides an additional rate for the convergence of the time-averaged kinetic energy
as $t\to\infty$.
\end{remark}

\section{Relative entropy inequality}
\label{sec:relent}

In this section, we consider weak solutions $(\rho,n)$ 
satisfying~\ref{item:E1} or~\ref{item:E2} and derive an inequality for the relative entropy
$\EE(\tau;\rho,n \, \vert \, \ol\rho,\ol n)$
defined in~\eqref{eq:relentr.relentflux},
where $(\ol\rho,\ol n)$ is a suitable reference pair.
To this end, we define the residuals
\begin{equation}
\begin{aligned}
    \ol R_1
    &:=\ol\rho_\tau -\frac{y}{2}\ol\rho_y+\ol n_y,
    \\
    \ol R_2
    &:=\ol n_\tau - \frac{y}{2}\ol n_y - \frac12 \ol n + \Big( \frac{\ol n^2}{\ol\rho} \Big)_y 
    + \ee^\tau \big( \pp(\ol\rho)_y + \alpha \ol n\big),
\end{aligned}
\label{eq:residuals}
\end{equation}
and the error term
$\xi := \xi_1+\xi_2+\xi_3$ with
    \begin{equation}
    \begin{aligned}
        \xi_1 &:=  - \Big( \frac{\ol n}{\ol\rho}\Big)_y\big( \ee^{-\tau} \rho \Big\vert \frac n\rho -\frac{\ol n}{\ol\rho} \Big\vert^2 + \pp (\rho \, \vert \, \ol \rho) \big) , \\
        \xi_2 &:=   \ee^{-\tau }  \ol R  \; \frac{\rho}{\ol\rho} \Big(\frac n\rho -\frac{\ol n}{\ol\rho}\Big), \\
        \xi_3 &:= -(\rho-\ol\rho) h''(\ol\rho) \ol R_1.
    \end{aligned}
    \label{eq:xi.def}
    \end{equation}
Here, $\pp (\rho \, \vert \, \ol \rho) := \pp(\rho) - \pp(\ol \rho) - \pp'(\ol \rho) (\rho {-} \ol \rho)$ is the relative pressure and $\ol R:= \frac{\ol n}{\ol\rho} \ol R_1- \ol R_2$.
We further set
\begin{equation}
  \label{eq:frict.diss}
  \damp(\rho,n\,\vert\,\ol\rho,\ol n) := \alpha \int_\R \rho \,\Big\vert \frac n\rho - \frac{\ol n}{\ol\rho} \Big\vert^2 \dd y
  \quad\text{and}\quad
  \Xi = \sum_{j=1}^3 \Xi_j
\end{equation}
with $\Xi_j:=\int_\R\xi_j\dd y$,  $j=1,2,3$.
We shall show the relative entropy inequality
\begin{align}
   \begin{aligned}
   \label{eq:relent.total.weak}
   \forall\, \varphi\in C^1_c([0,\infty);[0,\infty)):&
   \\
    \int_0^\infty -\varphi'\, \EE(\tau;\rho,n \, \vert \, \ol\rho,\ol n)  
    &+  \varphi \Big(\frac12\EE(\tau;\rho,n \, \vert \, \ol\rho,\ol n) 
    + \damp(\rho,n\,\vert\,\ol\rho,\ol n) \Big)\dd \tau
    \\
     &\quad
     \leq \varphi(0)\,\EE(0;\rho_0,n_0 \, \vert \, \ol\rho(0),\ol n(0))
     +\int_0^\infty \varphi \, \Xi \dd\tau.
   \end{aligned}  
\end{align}
Before we give separate proofs under the assumptions~\ref{item:E1} and~\ref{item:E2},
we collect some results useful for the derivation of~\eqref{eq:relent.total.weak}.

\subsection{Preparatory results}
\label{sec:preparatory}

In what follows, we prove several lemmata.
In order to treat the case of vacuum for the reference density, 
we first derive additional properties on $\pp$ and $h$ 
if the pressure satisfies~\eqref{eq:pres.better}.

\begin{lemma}\label{lem:pres.better}
Let $\pp\in C([0,\infty))\cap C^1(0,\infty)$ with $\pp'>0$ in $(0,\infty)$,
and such that~\eqref{eq:pres.better} holds.
Then $h\in C^1([0,\infty))$ with $h'(0)=0$, 
and $
\rho h'(\rho)\leq c h(\rho\,\vert\,0)$
with $c$ as in~\eqref{eq:pres.better}.
\end{lemma}
\begin{proof}
    We can compute $h$ as a solution to the ordinary differential equation in~\eqref{eq:h.properties}, which gives the representation
    \[
    h(\rho)=\rho\int_{\rho_0}^\rho\frac{p(z)}{z^2}\dd z
    \]
    for any $\rho_0>0$. 
    Due to~\eqref{eq:pres.better}, we can send $\rho\to 0$,
    which yields 
    $h\in C^1([0,\infty))$ with $h'(0)=0$.
    By~\eqref{eq:pres.better} and~\eqref{eq:h.properties}, we can further estimate
    \[
    \rho h'(\rho)
    =\rho\int_{0}^\rho h''(z)\dd z
    = \rho\int_{0}^\rho\frac{p'(z)}{z}\dd z
    \leq c\rho\int_{0}^\rho\frac{p(z)}{z^2}\dd z
    =c\,h(\rho\,\vert\,0)
    \]
    which completes the proof.    
\end{proof}

The next lemma gives control 
of the relative entropy flux by the relative entropy, both defined in~\eqref{eq:relentr.relentflux}.

\begin{lemma}\label{lem:entropyflux.control}
Let $\pp\in C([0,\infty))\cap C^1(0,\infty)$ with $\pp'>0$ in $(0,\infty)$. 
Let $0 < \delta < M$. 
Then there exists $C>0$ such that 
\begin{equation}\label{eq:entropyflux.control}
|q(\tau;\rho_1,n_1\,\vert\,\rho_2,n_2)|
\leq C\Big(\frac{|n_1|}{\rho_1}+\frac{|n_2|}{\rho_2}+\ee^{\tau/2}\Big)\eta(\tau;\rho_1,n_1\,\vert\,\rho_2,n_2)
\end{equation}
for all $\tau\geq 0$, 
all $(\rho_1,n_1)\in\{(0,0)\}\cup (0,M]\times\R$
and all $(\rho_2,n_2)\in [\delta,M]\times\R$. 
If $\pp$ additionally satisfies~\eqref{eq:pres.better},
the statement
also holds for $(\rho_2,n_2)= (0,0)$.
\end{lemma}

\begin{proof}
    First consider $(\rho_1,n_1)=(0,0)$ and $(\rho_2,n_2)\in [\delta,M]\times\R$.
    Then we have
    \[
    |q(\tau;0,0\,\vert\,\rho_2,n_2)|
    =\frac{|n_2|}{\rho_2}\big|-h(\rho_2)+h'(\rho_2)\rho_2\big|
    =\frac{|n_2|}{\rho_2}\eta(\tau;0,0\,\vert\,\rho_2,n_2),
    \]
    which shows~\eqref{eq:entropyflux.control} in this case.
    
    Now let $\rho_1>0$ and $\rho_2\in(\delta,M)$.
    We obtain
    \[
    \begin{aligned}
    |q(\tau;\rho_1,n_1\,\vert\,\rho_2,n_2)|
    & \leq 
    \max\Big\{\frac{|n_1|}{\rho_1},\frac{|n_2|}{\rho_2}\Big\}\eta(\tau;\rho_1,n_1\,\vert\,\rho_2,n_2)
    + \rho_1 \big|  h'(\rho_1) {-}  h'(\rho_2) \big|  \Big\vert \frac{n_1}{\rho_1} -\frac{n_2}{\rho_2} \Big\vert.
    \end{aligned}
    \]
To estimate the latter term,
consider first the case $0 <\rho_1 < \delta/2$.
We use~\eqref{eq:h.properties} to deduce 
$\rho_1 \big(  h'(\rho_1) -  h'(\rho_2) \big) = h(\rho_1\,\vert\,\rho_2) + \pp(\rho_1)-\pp(\rho_2)$.
First, we have $\big| \pp(\rho_1)-\pp(\rho_2) \big| \leq  \pp(M) - \pp(\delta/2)$ by monotonicity of $\pp$. 
Second, we use Taylor's theorem and $h''(z) >0$ for $z>0$ to conclude
    \[
    h(\rho_1\,\vert\,\rho_2)
    = - \int_{\rho_1}^{\rho_2}(\rho_1{-}z)h''(z)\dd z 
    \geq   \int_{\delta/2}^{\delta} \big(z-\frac\delta2 \big) h''(z)\dd z \geq \frac{\delta^2}{8} \min_{z\in[\delta/2,\delta]} h''(z).
    \]
In sum, this gives
\[
\rho_1 \big| h'(\rho_1) {-}  h'(\rho_2) \big|  \Big| \frac{n_1}{\rho_1} -\frac{n_2}{\rho_2} \Big|
\leq  c \Big( \frac{|n_1|}{\rho_1} + \frac{|n_2|}{\rho_2}\Big) h(\rho_1\,\vert\,\rho_2)
\]
for some $c >0$ in this case.
    If $\rho_1\geq\delta/2$, we instead estimate
    $|h'(\rho_1) -  h'(\rho_2)|
    \leq\max_{z\leq M}h''(z) |\rho_1-\rho_2|$
    and
    $h(\rho_1\,\vert\,\rho_2)
    \geq \frac{1}{2}\min_{z\in[\delta/2,M]}h''(z)\,(\rho_2-\rho_1)^2$,
    which together with Young's inequality implies
    \[
    \begin{aligned}
    \rho_1 \big|  h'(\rho_1) -  h'(\rho_2) \big|  \Big\vert \frac{n_1}{\rho_1} -\frac{n_2}{\rho_2} \Big\vert
    &\leq\ee^{\tau/2} \,\frac{1}{2}\ee^{-\tau}\rho_1\Big\vert \frac{n_1}{\rho_1} -\frac{n_2}{\rho_2} \Big\vert^2
    + \frac{1}{2} \ee^{\tau/2}\rho_1 \,\big|h'(\rho_1) -  h'(\rho_2)\big|^2
    \\
    &\leq c\,\ee^{\tau/2} \eta(\tau;\rho_1,n_1\,\vert\,\rho_2,n_2)
    \end{aligned}
    \]
    for some $c>0$ since $\min_{z\in[\delta/2,M]} h''(z)> 0$.
    Collecting all estimates, we arrive at~\eqref{eq:entropyflux.control}.

    Finally,
    assume~\eqref{eq:pres.better} and
    let $(\rho_2,n_2)=(0,0)$ 
    and $(\rho_1,n_1)\in\{(0,0)\}\cup (0,M]\times\R$.
    Then $q(\tau;\rho_1,n_1\,\vert\,0,0)$ is well defined 
    by Lemma~\ref{lem:pres.better}, and we obtain
    \[
    |q(\tau;\rho_1,n_1\,\vert\,0,0)|
    = \frac{|n_1|}{\rho_1}\big|\ee^{-\tau}\frac{n_1^2}{2\rho_1}+\rho_1 \big(h'(\rho_1)-h'(0)\big)\big|
    \leq (1+c)\frac{|n_1|}{\rho_1}\eta(\tau;\rho_1,n_1\,\vert\,0,0)
    \]
    for $c>0$ as in~\eqref{eq:pres.better}, which yields~\eqref{eq:entropyflux.control} in this case.
\end{proof}

The following auxiliary results focus on the case of a barotropic pressure law of
the form~\eqref{eq:pres.gamma}. To begin with, we
introduce the family of functions
\begin{align}
\label{eq:Fp}
F_p(z) := {
   \begin{cases} 
	\frac{1}{p (p{-}1)} \big( z^p - p z+ p -1\big) 
		&\text{for } p \in \mathbb{R} \setminus \{0,1 \},\\
	z \log z -z+1 &\text{for } p=1, \\
	z-\log z-1  &\text{for } p =0,
   \end{cases}}
\end{align}
which are determined by the conditions $F_p''(z) = z^{p{-}2}$ and $F_p(1) = F_p'(1)=0$. 
These are often used as entropy functions 
for reaction-diffusion systems (see e.g.\ \cite{MieMit18CEER, MieSch24CSSP}),
and they satisfy the lower bound
\begin{equation}
\label{eq:Fp.bound}
F_p(z) \geq \frac{1/2} {\max\{p,1{-}p\}}\, F_{1/2}(z) = \frac{1}{\max\{p,1{-}p\}} \, (\sqrt{z}-1)^2 
\end{equation}
for all $p>0$ and $z>0$,
see \cite[eqn.\,(3.2)]{MieMit18CEER}.
If the potential energy $h$ is of the form \eqref{eq:h_gammaLaw}, then we have
\begin{equation}
    \label{eq:relH.Fp}
 h(\rho \, \vert \, \ol\rho) = \gamma k F_\gamma(\rho / \ol\rho) {\ol\rho}^\gamma \quad \text{ for all } \gamma \geq 1,
\end{equation}
so that
\begin{equation}
    \label{eq:relH.est.sqrt}
 h(\rho \, \vert \, \ol\rho) \geq k  {\ol\rho}^{\gamma-1} \big(\sqrt{\rho}-\sqrt{\ol\rho}\big)^2 
\end{equation}
by~\eqref{eq:Fp.bound}.

In the following auxiliary result, we provide conditions that allow to exchange arguments in the relative entropy.
We state the result in a time-independent and a time-dependent version.

\begin{proposition}
    \label{prop:relent.difference}
    Let $\pp$ be given as in~\eqref{eq:pres.gamma}.
    \begin{enumerate}[label=\roman*.]
    \item 
    Let $\tau\in[0,\infty)$, and let $(\rho_1,n_1),(\rho_2,n_2)\in L^1_\mathrm{loc}(\R)^2$.
    Assume that 
    $\rho_1,\frac{1}{\rho_1},\frac{n_1}{\rho_1}\in L^\infty(\R)$ and
    \[
    \Big(\frac{n_1}{\rho_1}-\frac{n_2}{\rho_2}\Big),
    \
    \Big(\frac{n_1^2}{\rho_1^2}-\frac{n_2^2}{\rho_2^2}\Big),
    \ h'(\rho_1)-h'(\rho_2)
    \, \in\,  L^2(\R)\cap L^\infty(\R).
    \]
    If $(\rho,n)\in L^1_\mathrm{loc}(\R)^2$
    such that
    $\EE(\tau;\rho,n\,\vert\,\rho_1,n_1)<\infty$,
    then
    $\EE(\tau;\rho,n\,\vert\,\rho_2,n_2)<\infty$
    if and only if $\EE(\tau;\rho_1,n_1\,\vert\,\rho_2,n_2)<\infty$.
    \item 
    Let $I\subset [0,\infty)$ be an interval,
    and let $(\rho_1,n_1),(\rho_2,n_2)\in L^1_\mathrm{loc}(I\times\R)^2$.
    Assume that 
    $\rho_1,\frac{1}{\rho_1},\ee^{-\tau/2}\frac{n_1}{\rho_1}\in L^\infty(I\times\R)$ and
    \[
    \ee^{-\tau/2}\Big(\frac{n_1}{\rho_1}-\frac{n_2}{\rho_2}\Big),
    \
    \ee^{-\tau}\Big(\frac{n_1^2}{\rho_1^2}-\frac{n_2^2}{\rho_2^2}\Big),
    \  h'(\rho_1)-h'(\rho_2)
    \, \in\,  L^2(I\times\R)\cap L^\infty(I\times\R).
    \]
    If $(\rho,n)\in L^1_\mathrm{loc}(I\times\R)^2$ such that 
    $\EE(\tau;\rho,n\,\vert\,\rho_1,n_1)\in L^1(I)$,
    then
    $\EE(\tau;\rho,n\,\vert\,\rho_2,n_2)\in L^1(I)$
    if and only if $\EE(\tau;\rho_1,n_1\,\vert\,\rho_2,n_2)\in L^1(I)$.
    \end{enumerate}
\end{proposition}
\begin{proof}
    A straightforward calculation shows
    \begin{equation}
    \label{eq:eta.difference}
    \begin{aligned}
    &\eta(\tau;\rho,n\,\vert\,\rho_1,n_1)
    +\eta(\tau;\rho_1,n_1\,\vert\,\rho_2,n_2) 
    -\eta(\tau;\rho,n\,\vert\,\rho_2,n_2)
    \\
    &\quad
    =-\ee^{-\tau}\Big(\frac{n_1}{\rho_1}-\frac{n_2}{\rho_2}\Big)(n-n_1)
    -\bigg(\frac{1}{2}\ee^{-\tau}\Big(\frac{n_1^2}{\rho_1^2}-\frac{n_2^2}{\rho_2^2}\Big)+h'(\rho_1)-h'(\rho_2)\bigg)(\rho-\rho_1).
    \end{aligned}
    \end{equation}
    The claim follows if the right-hand side is in $L^1(Q)$,
    where $Q=\R$ or $Q=I\times\R$, 
    depending on the case.
    Inequality~\eqref{eq:relH.est.sqrt} implies
    \[
    \rho-\rho_1=(\sqrt{\rho}-\sqrt{\rho_1})^2
    + 2\sqrt{\rho_1}(\sqrt{\rho}-\sqrt{\rho_1})
    \in L^1(Q)+L^2(Q),
    \]
    and similarly,
    the decomposition
    \[
    \ee^{-\tau/2}(n-n_1) 
    = (\sqrt{\rho}-\sqrt{\rho_1}) \ee^{-\tau/2}\sqrt{\rho}\Big(\frac{n}{\rho}-\frac{n_1}{\rho_1}\Big)
    +\sqrt{\rho_1}\ee^{-\tau/2}\sqrt{\rho}\Big(\frac{n}{\rho}-\frac{n_1}{\rho_1}\Big)
    +\ee^{-\tau/2}\frac{n_1}{\rho_1}(\rho-\rho_1)
    \]
    yields $\ee^{-\tau/2}(n-n_1)\in L^1(Q)+L^2(Q)$.
    In view of the assumptions,
    we conclude the integrability of the right-hand side of~\eqref{eq:eta.difference}
    in both cases, 
    which completes the proof.
\end{proof}

We further need the following coercivity estimate of the relative entropy.

\begin{lemma}\label{lem:coercivity}
    Let $\pp$ satisfy~\eqref{eq:pres.gamma}.
    For $\delta,M>0$ define the interval $K=[\delta,M]$.
    Then there exist constants $r_0>0$ and $C>0$, depending only on $K$, such that
    \begin{equation}\label{eq:eta.coercive}
    \eta(\tau;\rho,n\,\vert\,\ol\rho,0)
    \geq 
    \begin{cases}
    C\big(\ee^{-\tau}|n|^2+|\rho-\ol\rho|^2\big) &\text{ if }\rho\in[0,r_0],
    \\
    C\big((\ee^{-\tau}|n|^2)^{\gamma/(\gamma+1)} + |\rho-\ol\rho|^\gamma\big) &\text{ if }\rho\in(r_0,\infty),
    \end{cases}
    \end{equation}
    for all $n\in\R$ and  $\ol\rho\in K$.
    Moreover, if $\ol\rho=0$, then \eqref{eq:eta.coercive} holds with $r_0=0$.
\end{lemma}

\begin{proof}
    For given $K=[\delta,M]$, let $r_0>0$.
    We first show that there exists $c> 0$
    such that
    \begin{equation}\label{eq:hrel.est.h}
    h(\rho\,\vert\,\ol \rho)
    \geq 
    \begin{cases}
    c\,|\rho-\ol\rho|^2 &\text{ if }\rho\leq r_0,
    \\
    c\,h(\rho-\ol\rho) &\text{ if }\rho> r_0,
    \end{cases}
    \end{equation}
    for all $\ol\rho\in K$.
    To this end, we show that 
    \[
    (\rho,\ol\rho)\mapsto 
    \frac{|\rho-\ol\rho|^2}{h(\rho\,|\,\ol\rho)} \chi_{[0,r_0]}(\rho)
    + \frac{h(\rho-\ol\rho)}{h(\rho\,|\,\ol\rho)} \chi_{(r_0,\infty)}(\rho)
    \]
    defines a bounded function on $(0,\infty)\times K$,
    where $\chi_M$ is the characteristic function of a set $M\subset\R$.
    Indeed, by Taylor's formula, we have for $\rho\in K$ that
    \[
    h(\rho\,\vert\,\ol\rho)
    =\int_{\ol\rho}^\rho h''(z)(\rho-z)\dd z 
    \geq \frac{1}{2}\min_{z\in K}h''(z) (\rho-\ol\rho)^2
    = c_1(\rho-\ol\rho)^2 
    \]
    for some $c_1>0$ due to $h''(z)=\frac{\pp'(z)}{z}>0$.
    Since $h(z)/z\to\infty$ as $z\to\infty$, we obtain
    \[
    h(\rho\,\vert\,\ol\rho)
    \geq h(\rho)-\max_{z\in K} h(z)- \max_{z\in K} h'(z)\rho + \min_{z\in K} h'(z) z
    \geq c_2 h(\rho)
    \]
    for some constant $c_2>0$
    if $\rho$ is sufficiently large. 
    In total, this yields~\eqref{eq:hrel.est.h} 
    for any choice $r_0\geq M$.
    In particular, we can choose $r_0\geq 2M$.
    Note that~\eqref{eq:hrel.est.h} holds for all $\gamma\geq 1$, 
    and using~\eqref{eq:h_gammaLaw}
    we can conclude the existence of $r_0\geq 2M$ and $c>0$ such that
    \begin{equation}\label{eq:hrel.est.gamma}
    h(\rho\,\vert\,\ol \rho)
    \geq 
    \begin{cases}
    c\,|\rho-\ol\rho|^2 &\text{ if }\rho\leq r_0,
    \\
    c\,|\rho-\ol\rho|^\gamma &\text{ if }\rho> r_0,
    \end{cases}
    \end{equation}
    for all $\ol\rho\in K$.   
    Clearly,~\eqref{eq:hrel.est.h} and~\eqref{eq:hrel.est.gamma} also hold
    for $r_0=0$ if $\ol\rho=0$.
    To estimate the kinetic-energy term, 
    we use~\eqref{eq:hrel.est.gamma} to first derive
    \[
    \eta(\tau;\rho,n\,\vert\,\ol\rho,0)
    \geq \ee^{-\tau}\frac{|n|^2}{2r_0}+c|\rho-\ol\rho|^2 
    \qquad
    \text{if }\rho\leq r_0.
    \]
    By Young's inequality,
    we further have
    \[
    \ee^{-\tau}\frac{|n|^2}{2\rho}+\frac{c}{2}|\rho-\ol\rho|^\gamma
    \geq C_\gamma (\ee^{-\tau}|n|^2)^{\gamma/(\gamma+1)}\Big(\frac{|\rho-\ol\rho|}{\rho}\Big)^{\gamma/(\gamma+1)}
    \]
    for some constant $C_\gamma>0$.
    Since $|\rho-\ol\rho|\geq \rho - M \geq \rho-r_0/2$, we conclude with~\eqref{eq:hrel.est.gamma} that
    \[
    \eta(\tau;\rho,n\,\vert\,\ol\rho,0)
    \geq 
    C_\gamma (\ee^{-\tau}|n|^2)^{\gamma/(\gamma+1)}\Big(\frac{1}{2}\Big)^{\gamma/(\gamma+1)}
    +\frac{c}{2}|\rho-\ol\rho|^\gamma
    \qquad
    \text{if }\rho> r_0.
    \]
    In total, this shows~\eqref{eq:eta.coercive}.
\end{proof}

In the following lemmata, we derive useful estimates of the functions $\xi_1$, $\xi_2$, $\xi_3$ defined in~\eqref{eq:xi.def}
in the case of a pressure law of the form~\eqref{eq:pres.gamma}.

\begin{lemma} \label{le:xi.1.bound}
    Assume that $\pp$ satisfies~\eqref{eq:pres.gamma}. 
    Let $(\rho,n)\in L^1_\mathrm{loc}((0,\infty)\times\R)^2$ such that $\rho\geq 0$ a.e.,
    and let $(\ol\rho,\ol n)\in C^1((0,\infty)\times\R)^2$ such that $\inf\ol\rho>0$.
    Then $\xi_1$ satisfies the estimates
 \[
 \begin{aligned}
   \xi_1 
   &\leq \max\{2,\gamma-1 \}\Big[-\Big(\frac{\ol n}{\ol \rho}\Big)_y \Big]_+\, \eta(\tau; \rho,n \, | \, \ol \rho, \ol n),
   \\
   |\xi_1| 
   &\leq \max\{2,\gamma-1 \}\Big|\Big(\frac{\ol n}{\ol \rho}\Big)_y \Big| \ \eta(\tau; \rho,n \, | \, \ol \rho, \ol n).
   \end{aligned}
   \]
\end{lemma}
\begin{proof}
We use $\pp(\rho \, \vert \, \ol \rho) = (\gamma{-}1) \, h(\rho \, \vert \, \ol \rho)$, which holds for all $\gamma \geq 1$, and immediately obtain 
both estimates since the relative entropy is nonnegative.
\end{proof}

\begin{lemma} \label{le:xi.2.bound}
Under the assumptions of Lemma~\ref{le:xi.1.bound}, 
the function $\xi_2$ satisfies the estimate
\[
| \xi_2 | \leq \Big(
\frac{1}{2 k \ol\rho^\gamma} + \frac{3}{2\ol \rho}
\Big)|\ol R|
\, \ee^{-\tau/2} \, \eta(\tau; \rho,n \, | \, \ol \rho, \ol n) +  \ee^{-\tau/2} |\ol R|, 
\]
where $\ol R:= \frac{\ol n}{\ol\rho} \ol R_1- \ol R_2$
with $\ol R_1$, $\ol R_2$ defined in~\eqref{eq:residuals}.
\end{lemma}

\begin{proof}
    We split $\xi_2$ into 
    \begin{align*}
    \xi_2 
    &= \ee^{-\tau} \frac{\ol R}{\ol \rho} \Big( \frac{n}{\rho} - \frac{\ol n}{\ol \rho}\Big) \sqrt{\rho} \big( \sqrt{\rho} {-} \sqrt{\ol \rho} \big) 
    +  \ee^{-\tau} \frac{\ol R}{\sqrt{\ol \rho}} \sqrt{\rho}\Big( \frac{n}{\rho} - \frac{\ol n}{\ol \rho}\Big) .
    \end{align*}
By Young's inequality and~\eqref{eq:relH.est.sqrt}, the first term gives
\begin{align*}
     \Big|\ee^{-\tau} & \frac{\ol R}{\ol \rho} \Big( \frac{n}{\rho} - \frac{\ol n}{\ol \rho}\Big) \big( \sqrt{\rho} {-} \sqrt{\ol \rho} \Big) \sqrt{\rho} \Big|
    \leq \ee^{-\tau/2} \frac{|\ol R|}{\ol \rho} \Big[ \frac12 \big( \sqrt{\rho} {-} \sqrt{\ol \rho} \big)^2  + \frac12 \ee^{-\tau}  \rho \Big( \frac{n}{\rho} - \frac{\ol n}{\ol \rho}\Big)^2  \Big] \\
    & \leq \ee^{-\tau/2} \frac{|\ol R|}{\ol \rho} \Big[ \frac{1}{2k} {\ol \rho}^{1-\gamma}  h(\rho \, | \, \ol \rho)  + \frac12 \ee^{-\tau}  \rho \Big( \frac{n}{\rho} - \frac{\ol n}{\ol \rho}\Big)^2  \Big]    
    \leq \ee^{-\tau/2} \Big(\frac{|\ol R|}{2k\ol \rho^\gamma} +\frac{|\ol R|}{\ol\rho}\Big)
    \eta(\tau; \rho,n \, | \, \ol \rho, \ol n).
\end{align*}
For the second term, we use Young's inequality and obtain
\[
\Big|\ee^{-\tau} \frac{\ol R}{\sqrt{\ol \rho}} \Big( \frac{n}{\rho} - \frac{\ol n}{\ol \rho}\Big)  \sqrt{\rho} \Big|
\leq \ee^{-\tau/2} \Big[ |\ol R| + \frac{|\ol R|}{4\ol\rho} \ee^{-\tau} \rho \Big( \frac{n}{\rho} - \frac{\ol n}{\ol \rho}\Big)^2 \Big] 
\leq \ee^{-\tau/2} \Big[ |\ol R| +  \frac{|\ol R|}{2\ol\rho}\eta(\tau; \rho,n \, | \, \ol \rho, \ol n) \Big].
\]
Combining both estimates yields the desired result.
\end{proof}

\begin{lemma} \label{le:xi.3.bound}
    Under the assumptions of Lemma~\ref{le:xi.1.bound}, the function $\xi_3$ satisfies the estimate
    \[
   | \xi_3 | \leq \frac{2\gamma |\ol R_1|} {{\ol\rho}} \, h( \rho \, | \, \ol \rho) + \ol \rho | h''(\ol \rho) \ol R_1 |
    \]
    with $\ol R_1$ defined in~\eqref{eq:residuals}.
\end{lemma}

\begin{proof}
    Analogously to the previous lemma, we use inequality~\eqref{eq:relH.est.sqrt}
    and Young's inequality to obtain 
    \begin{align*}
    | \xi_3 |
    &=  \big| (\rho - \ol \rho) h''(\ol \rho) \ol R_1 \big| = \big| \big(\sqrt{\rho} {-} \sqrt{\ol \rho}\big)^2
    + 2\sqrt{\ol \rho}\big(\sqrt{\rho} {-} \sqrt{\ol \rho}\big) \big | \, | h''(\ol \rho) \ol R_1 | \\
    & \leq  2\big(\sqrt{\rho} {-} \sqrt{\ol \rho}\big)^2 | h''(\ol \rho) \ol R_1|
    +  \ol \rho  | h''(\ol \rho) \ol R_1 | 
    \\
    &\leq  2| h''(\ol \rho) \ol R_1| \, \frac{1}{k}\ol \rho^{1-\gamma} h( \rho \, | \, \ol \rho)
    +  \ol \rho  | h''(\ol \rho) \ol R_1 | 
    = \frac{2\gamma |\ol R_1|} {{\ol\rho}} h( \rho \, | \, \ol \rho) + \ol \rho | h''(\ol \rho) \ol R_1 |
    \end{align*}
    as claimed.
\end{proof}

\subsection{Weak solutions satisfying~\ref{item:E1}}

We begin with the derivation of 
a local relative entropy inequality
for weak solutions that satisfy criterion~\ref{item:E1}.

\begin{proposition}\label{prop:relent.weak.local}
    Let $\pp\in C([0,\infty))\cap C^1(0,\infty)$ with $\pp'>0$,
    and let $(\rho,n)$ be a weak solution 
    to \eqref{eq:Euler.trans} that satisfies~\ref{item:E1}.
    Let
    $(\ol\rho,\ol n)\in C^1((0,\infty)\times\R)^2\cap C([0,\infty)\times\R)^2$ with $\ol\rho>0$
    or $(\ol\rho,\ol n)\equiv(0,0)$.
    Then 
    \begin{equation}
   \label{eq:relen.weak.local}
   \begin{aligned}
   &\int_0^\infty\!\!\!\int_\R -  \chi_{\tau}\, \eta (\tau; \rho,n\,\vert\,\ol\rho,\ol n)
   + \chi_y\Big(\frac{y}{2}\eta (\tau; \rho,n\,\vert\,\ol\rho,\ol n) - q(\tau;\rho,\eta\,\vert\,\ol\rho,\ol n)\Big) \dd y \dd \tau
   \\
   &\quad
   + \int_0^\infty\!\!\!\int_\R 
   \chi\Big(\frac{1}{2}\eta (\tau; \rho,n\,\vert\,\ol\rho,\ol n) 
   + \alpha \rho \Big\vert \frac n\rho {-} \frac{\ol n}{\ol\rho} \Big\vert^2\Big)
   \dd y \dd\tau
   \\
   &\qquad
   \leq \int_\R \chi(0)\,\eta(0;\rho_0,n_0\,\vert\,\ol\rho(0),\ol n(0))\dd y
   + \int_0^\infty\!\!\!\int_\R
   \chi \, \xi
   \dd y \dd\tau
   \end{aligned}
   \end{equation}
   for all $\chi\in C_c^\infty([0,\infty) \times \R)$ with $\chi\geq 0$,
   where $\xi = \xi_1 + \xi_2 +\xi_3$ 
   is defined by~\eqref{eq:xi.def}. 
\end{proposition}

\begin{proof}
We combine the local entropy inequality~\eqref{eq:weakEnt.scaled} from \ref{item:E1} with the weak formulation in~\eqref{eq:conteq.weak} and~\eqref{eq:momeq.weak}
for suitable test functions.
Let $\chi\in C_c^\infty([0,\infty) \times \R)$.
Testing~\eqref{eq:momeq.weak} with $\psi(\tau,y) = -\chi(\tau,y)\,\ee^{-\tau} \frac{\ol n(\tau,y)}{\ol\rho(\tau,y)} $ gives
\begin{equation}\label{eq:relen.test1}
\begin{aligned}
  &\int_0^\infty\!\!\!\int_\R  \chi_\tau\ee^{-\tau} \frac{ n\ol n}{\ol\rho}
  +\chi_y\,\ee^{-\tau} \Big[ - \frac{y}{2} n + \frac{n^2}{\rho} + \ee^{\tau} \pp(\rho) \Big]  \frac{\ol n}{\ol\rho}\dd y\dd\tau
  +\int_\R \chi(0) \frac{n_0\ol n(0)}{\ol\rho(0)}\dd y
   \\
 &\ 
 =  \int_0^\infty\!\!\!\int_\R -\chi\ee^{-\tau}n\Big[\Big(\frac{\ol n}{\ol\rho}\Big)_\tau - \frac{\ol n}{\ol\rho}\Big] 
 - \chi\ee^{-\tau} \Big[ - \frac{y}{2} n + \frac{n^2}{\rho} + \ee^{\tau} \pp(\rho) \Big]  \Big( \frac{\ol n}{\ol\rho} \Big)_y
 +  \chi\alpha \frac{n \ol n}{\ol\rho}  \dd y \dd \tau .
\end{aligned}
\end{equation}
Taking $\phi(\tau,y) = \frac12 \,\chi(\tau,y) \ee^{-\tau} \big( \frac{\ol n (\tau,y) }{\ol\rho(\tau,y) } \big)^2 $ 
as a test function in~\eqref{eq:conteq.weak}, we arrive at
\begin{equation}\label{eq:relen.test2}
\begin{aligned}
    &\int_0^\infty\!\!\!\int_\R - \chi_\tau \ee^{-\tau} \rho \frac{\ol n^2}{2\ol\rho^2}
    -\chi_y \ee^{-\tau}\Big[-\frac{y}{2} \rho+n\Big] \frac{\ol n^2}{2\ol\rho^2}
    \dd y\dd\tau 
    -\int_\R \chi(0) \rho_0\frac{\ol n(0)^2}{2\ol\rho(0)^2}\dd y
    \\
     &\ 
     = \int_0^\infty\!\!\! \int_\R  \chi\ee^{-\tau}\rho  \Big[\frac{\ol n}{\ol\rho}\Big( \frac{\ol n}{\ol\rho}\Big)_\tau - \frac12\Big( \frac{\ol n}{\ol\rho}\Big)^2\Big]
     +  \chi\ee^{-\tau} \Big[-\frac{y}{2} \rho+n\Big]  \frac{\ol n}{\ol\rho}\Big( \frac{\ol n}{\ol\rho}\Big)_y
     - \frac14 \chi \ee^{-\tau} \rho  \Big( \frac{\ol n}{\ol\rho}\Big)^2  \dd y \dd \tau .
\end{aligned}
\end{equation}
Testing equation \eqref{eq:conteq.weak} with  
$\phi(\tau,y) =  -\chi(\tau,y)\, h'(\ol\rho(\tau,y))$ instead, we obtain
\begin{equation}\label{eq:relen.test3}
\begin{aligned}
   &\int_0^\infty\!\!\!\int_\R  \chi_\tau h'(\ol\rho) \rho 
   + \chi_y  h'(\ol\rho) \Big[ - \frac{y}{2} \rho + n \Big]
   \dd y\dd\tau
   +\int_\R \chi(0) h'(\ol\rho(0))\rho_0\dd y
   \\
   &\quad
   =  \int_0^\infty\!\!\! \int_\R - \chi h''(\ol\rho) \ol\rho_\tau\rho
   - \chi h''(\ol\rho) \ol\rho_y \Big[ - \frac{y}{2} \rho + n \Big]
   + \frac12\chi h'(\ol\rho) \rho \dd y \dd \tau. 
\end{aligned}
\end{equation}
We further need the elementary identities
\[
\begin{aligned}
\int_0^\infty\!\!\!\int_\R 
\chi_\tau\big( h(\ol\rho)-h'(\ol\rho)\ol\rho\big) 
\dd y\dd\tau
&+\int_\R\chi(0)\big( h(\ol\rho(0))-h'(\ol\rho(0))\ol\rho(0)\big)\dd y
\\
&\qquad\qquad\qquad\qquad
=\int_0^\infty\!\!\!\int_\R \chi  h''(\ol\rho)\ol\rho\,\ol\rho_\tau\dd y \dd \tau,
\\
\int_0^\infty\!\!\!\int_\R 
-\chi_y\frac{y}{2}\big( h(\ol\rho)-h'(\ol\rho)\ol\rho\big)\dd y\dd\tau
&=\int_0^\infty\!\!\!\int_\R -\chi \frac{y}{2}  h''(\ol\rho)\ol\rho\,\ol\rho_y
+\frac{1}{2}\chi\big( h(\ol\rho)-h'(\ol\rho)\ol\rho\big)
\dd y \dd \tau.
\end{aligned}
\]
Adding all these identities to
the weak local entropy inequality~\eqref{eq:weakEnt.scaled}
and taking into account that
\[
    \begin{aligned}
    \eta(\tau;\rho,n\,\vert\,\ol\rho,\ol n)
    &=\eta(\tau;\rho,n)-\ee^{-\tau}\frac{n\,\ol n}{\ol\rho}+\ee^{-\tau}\rho\frac{\ol n^2}{2\ol\rho^2}-h(\ol\rho)-h'(\ol\rho)(\rho{-}\ol\rho),
    \\
    q(\tau;\rho,n\,\vert\,\ol\rho,\ol n)
    &=q(\tau;\rho,n) - \ee^{-\tau}\frac{n^2\ol n}{\rho\ol\rho} + \ee^{-\tau}\frac{n\ol n^2}{2\ol \rho^2}-n h'(\ol\rho)
    +\frac{\ol n}{\ol \rho}\big(h(\rho){-}h(\ol\rho){-}h'(\rho)\rho{+}h'(\ol\rho)\ol\rho\big),
    \end{aligned}
\]
we obtain
\[
\begin{aligned}
    &\int_0^\infty\!\!\!\int_\R -  \chi_{\tau}\, \eta (\tau; \rho,n\,\vert\,\ol\rho,\ol n)
   + \chi_y\Big(\frac{y}{2}\eta (\tau; \rho,n\,\vert\,\ol\rho,\ol n) - q(\tau;\rho,\eta\,\vert\,\ol\rho,\ol n)\Big) \dd y \dd \tau
   \\
   &\qquad
   -\int_\R \chi(0) \eta(0;\rho_0,n_0\,\vert\,\ol\rho(0),\ol n(0))\dd y
   \\
   &
   \leq 
   \int_0^\infty\!\!\!\int_\R 
   \chi\Big(-\frac{1}{2}\eta (\tau; \rho,n\,\vert\,\ol\rho,\ol n) 
   - \alpha \rho \Big\vert \frac n\rho {-} \frac{\ol n}{\ol\rho} \Big\vert^2
   +\xi \Big)
   \dd y \dd\tau
   \\
   &\qquad+\int_0^\infty\!\!\!\int_\R -\chi_y \frac{\ol n}{\ol\rho}\big(\pp(\rho)+h(\rho){-}h(\ol\rho){-}h'(\rho)\rho{+}h'(\ol\rho)\ol\rho\big) \dd y\dd\tau
   \\
   &\qquad +\int_0^\infty\!\!\!\int_\R \chi \Big(
   \rho\Big(\frac{n}{\rho}-\frac{\ol n}{\ol\rho}\Big)\frac{\pp(\ol\rho)_y}{\ol\rho}
   -\big(
   \pp(\ol\rho)+\pp'(\ol\rho)(\rho-\ol\rho)\big)\Big( \frac{\ol n}{\ol\rho} \Big)_y
   \Big)
   \dd y\dd\tau
   \\
   &\qquad
   +\int_0^\infty\!\!\! \int_\R  \chi h''(\ol\rho) \big(\ol n_y(\rho-\ol\rho)
   - \ol\rho_y n \big)
   \dd y \dd \tau. 
\end{aligned}
\]
Employing the identities  from~\eqref{eq:h.properties},
we can simplify the last three lines as
\[
\begin{aligned}
   &\int_0^\infty\!\!\!\int_\R -\chi_y \frac{\ol n}{\ol\rho}\pp(\ol\rho) 
   -\chi \Big(
   \rho\frac{\ol n}{\ol\rho^2}\pp'(\ol\rho)\ol\rho_y
   +\big(
   \pp(\ol\rho)+\pp'(\ol\rho)(\rho-\ol\rho)\big)\Big( \frac{\ol n}{\ol\rho} \Big)_y
   - \pp'(\ol\rho) \frac{\ol n_y}{\ol\rho}(\rho-\ol\rho)\Big)
   \dd y \dd \tau
   \\
   &\qquad
   =\int_0^\infty\!\!\!\int_\R \Big(-\chi\frac{\ol n}{\ol\rho}\pp(\ol\rho)\Big)_y\dd y \dd \tau=0,
\end{aligned}
\]
and we conclude~\eqref{eq:relen.weak.local}.
\end{proof}

To derive an inequality for the total relative entropy,
we make use of Lemma~\ref{lem:entropyflux.control}.

\begin{theorem}\label{thm:relentr.weak.0}
    Let $\pp\in C([0,\infty))\cap C^1(0,\infty)$ with $\pp'>0$,
    and let $(\rho,n)$ be a weak solution 
    to \eqref{eq:Euler.trans}
    fulfilling \ref{item:E1}
    with initial data $(\rho_0,n_0)\in L^1_{\mathrm{loc}}(\R)^2$.
    Let either $(\ol\rho,\ol n)\in C^1((0,\infty)\times\R)^2\cap C([0,\infty)\times\R)^2$ be bounded
    such that $\inf\ol\rho>0$,   
    or let $(\ol\rho,\ol n)\equiv(0,0)$ and assume that $\pp$ satisfies~\eqref{eq:pres.better}.
    Let
    \begin{equation}
    \label{eq:relentr.weak.E1.cond}
    \EE(0;\rho_0,n_0 \, \vert \, \ol\rho(0),\ol n(0))<\infty,
    \qquad
    \xi \in L^1_{\mathrm{loc}}([0,\infty);L^1(\R)).
    \end{equation}
    Then the relative entropy inequality~\eqref{eq:relent.total.weak} holds.
\end{theorem}

\begin{proof}
   One readily shows that~\eqref{eq:relen.weak.local}
   also holds if $\chi\geq 0$ is merely Lipschitz continuous
   with compact support in $[0,\infty) \times \R$.
   We can thus choose $\chi(\tau,y)=\varphi(\tau) \psi(\tau,y)$,
   where $\varphi\in C^1_c([0,\infty))$ with $\varphi\geq 0$
   and  $\supp\varphi\subset[0,T)$,
   and
   \begin{equation}
   \label{eq:testfct.special}
   \psi(\tau,y)=
       \begin{cases}
           1 & \text{if } |y|\leq  M\ee^{M(T-\tau)}-1 , \\
           M\ee^{M(T-\tau)}-|y| &\text{if }y\in A,\\
           0 & \text{if }|y|\geq M\ee^{M(T-\tau)},
       \end{cases}
   \end{equation}
   for $M>1$, where
   $A:=\{y\in \R \mid M\ee^{M(T-\tau)}-1<|y|<M\ee^{M(T-\tau)}\}$.
   We set $\ol\eta(\tau)=\eta (\tau; \rho(\tau),n(\tau)\,\vert\,\ol\rho(\tau),\ol n(\tau)) $
   and $\ol q(\tau)=q (\tau; \rho(\tau),n(\tau)\,\vert\,\ol\rho(\tau),\ol n(\tau))$.
   Then~\eqref{eq:relen.weak.local} reduces to
   \begin{equation}
   \begin{aligned}
   &\int_0^\infty\!\!\!\int_\R -  \varphi_{\tau}\psi\, \ol\eta  
   +\varphi\psi\Big(\frac{1}{2}\ol\eta 
   + \alpha \rho \Big\vert \frac n\rho {-} \frac{\ol n}{\ol\rho} \Big\vert^2\Big)
   \dd y\dd\tau
   \\
   &\quad+ \int_0^\infty\!\!\!\int_{A}\varphi \Big(
   M^2\ee^{M(T-\tau)}\ol\eta 
   -\frac{y}{|y|}\frac{y}{2}\ol\eta 
   +\frac{y}{|y|} \ol q \Big)\dd y \dd \tau
   \\
   &\qquad
   \leq \int_\R \varphi(0)\psi(0,y)\,\ol\eta(0)\dd y
   + \int_0^\infty\!\!\!\int_\R
   \varphi\psi \, \xi
   \dd y \dd\tau.
   \end{aligned}
   \label{eq:relen.weak.cutofftest}
   \end{equation}
   In virtue of Lemma~\ref{lem:entropyflux.control} 
   and the boundedness of $\frac{n}{\rho}$ and $\frac{\ol n}{\ol \rho}$, 
   we can choose $M>1$ so large that
   $| \ol q(\tau)|\leq \frac{M^2}{2} \ol\eta(\tau)$
   for $\tau\in(0,T)$.
   With this choice, the integral over $A$ in~\eqref{eq:relen.weak.cutofftest}
   is nonnegative and can be omitted.
   Passing to the limit $M\to\infty$ in the resulting inequality,
   we conclude~\eqref{eq:relent.total.weak}.
\end{proof}

\begin{remark}\label{rem:E1.implies.E2}
    The particular choice $(\ol\rho,\ol n)=(\rho_\infty,0)$ for 
    some $\rho_\infty \in C^1(\R)$ satisfying~\eqref{eq:rhoinfty} is admissible
    in Theorem~\ref{thm:relentr.weak.0},
    which results in the relative entropy inequality~\eqref{eq:weakEnt.total.rhoinf}.
    This shows that criterion~\ref{item:E1} implies criterion~\ref{item:E2}.
\end{remark}

In Theorem~\ref{thm:relentr.weak.0},
the reference pair $(\ol\rho,\ol n)$ has to satisfy~\eqref{eq:relentr.weak.E1.cond},
which involves relations between $(\ol\rho,\ol n)$ and the considered weak solution $(\rho,n)$.
In the case of a polytropic gas, we can also derive 
conditions independent of $(\rho,n)$, which are more preferable.

\begin{theorem}
    \label{thm:relentr.weak.E1}
    Let $\pp$ satisfy~\eqref{eq:pres.gamma},
    and let $(\rho,n)$ be a weak solution 
    to \eqref{eq:Euler.trans}
    fulfilling \ref{item:E1}
    with initial data $(\rho_0,n_0)\in L^1_{\mathrm{loc}}(\R)^2$
    such that $\EE(0;\rho_0,n_0\,\vert \,\hat\rho,0)<\infty$.
    Let $(\ol\rho,\ol n)\in C^1((0,\infty)\times\R)^2\cap C([0,\infty)\times\R)^2$
    such that $\inf\ol\rho>0$
    and 
    \begin{equation}\label{eq:cond.relentr.E1}
    \begin{aligned}
    \ol\rho,\,\ol n,\, \ol R_1, \,\ol R_2
    &\in L^\infty_{\mathrm{loc}}([0,\infty);L^1(\R)\cap L^\infty(\R)),
    \\
    \EE(\tau;\ol\rho,\ol n\,\vert \,\hat\rho,0)&\in L_\mathrm{loc}^\infty([0,\infty)).
    \end{aligned}
    \end{equation}
    Then the relative entropy inequality~\eqref{eq:relent.total.weak} holds.
\end{theorem}
\begin{proof}
    As observed in Remark~\ref{rem:E1.implies.E2},
    criterion~\ref{item:E2} is satisfied,
    so that
    $\EE(\tau;\rho, n\,\vert \,\rho_\infty,0)\in L_\mathrm{loc}^\infty([0,\infty))$ for some $\rho_\infty \in C^1(\R)$ satisfying~\eqref{eq:rhoinfty}.
    From~\eqref{eq:cond.relentr.E1} and Proposition~\ref{prop:relent.difference}
    we further obtain
    $\EE(\tau;\rho, n\,\vert \,\ol\rho,\ol n)\in L_\mathrm{loc}^\infty([0,\infty))$.
    Using now~\eqref{eq:cond.relentr.E1}
    and the estimates from Lemmas~\ref{le:xi.1.bound}, \ref{le:xi.2.bound} and \ref{le:xi.3.bound},
    we obtain $\xi \in L^1_{\mathrm{loc}}([0,\infty);L^1(\R))$.
    From the finiteness of $\EE(0;\rho_0,n_0\,\vert \,\hat\rho,0)$ and $\EE(0;\ol\rho(0),\ol n(0) \, \vert \, \hat\rho,0)$, we further conclude
    $\EE(0;\rho_0,n_0 \, \vert \, \ol\rho(0),\ol n(0))<\infty$
    by Proposition~\ref{prop:relent.difference}.
    Now the statement follows from Theorem~\ref{thm:relentr.weak.0}.        
\end{proof}

\subsection{Weak solutions satisfying~\ref{item:E2}}

We now address the relative entropy inequality for weak solutions that are subject to criterion~\ref{item:E2}.
First we consider spatially localized pairs $(\ol\rho,\ol n)$.

\begin{lemma} \label{le:relentr.E2.compact.supp}
    Let $\pp\in C([0,\infty))\cap C^1(0,\infty)$ satisfy $\pp'>0$.
    Let $(\rho,n)$ be a weak solution 
    to \eqref{eq:Euler.trans} fulfilling~\ref{item:E2}
    with initial data $(\rho_0,n_0)\in L^1_{\mathrm{loc}}(\R)^2$
    such that $\EE(0;\rho_0,n_0\,\vert \,\hat\rho,0)<\infty$.
    Let either $(\ol\rho,\ol n)\in C^1((0,\infty)\times\R)^2\cap C([0,\infty)\times\R)^2$ such that
    $\ol\rho> 0$,
     or let $(\ol\rho,\ol n)\equiv(0,0)$ and assume 
   that $\pp$ satisfies~\eqref{eq:pres.better}.
   Further assume that
   $\ol\rho-\hat\rho$ and $\ol n$ have compact support in space.
   Then the relative entropy inequality~\eqref{eq:relent.total.weak} holds.
\end{lemma}

\begin{proof}
Let $\varphi\in C^1_c([0,\infty))$ with $\varphi\geq 0$
and $\supp \varphi\subset [0,T)$,
and let $\rho_\infty$ be the reference density from~\ref{item:E2}.
Since $\ol\rho-\rho_\infty$ has compact support, we obtain the elementary identity
\begin{equation}
\label{eq:hrhoinf}
\begin{aligned}
&\int_0^\infty\!\!\!\int_{\R}-\varphi'\big( h(\rho_\infty)-h(\ol\rho)-h'(\rho_\infty)\rho_\infty+h'(\ol\rho)\ol\rho\big)\dd y \dd\tau
\\
&-\varphi(0)\int_\R h(\rho_\infty)-h(\ol\rho(0))-h'(\rho_\infty)\rho_\infty+h'(\ol\rho(0))\ol\rho(0)\dd y
=\int_0^\infty\!\!\!\int_\R \varphi h''(\ol\rho)\ol\rho\,\ol\rho_\tau\dd y \dd \tau.
\end{aligned}    
\end{equation}
Moreover, there is $M>0$ such that 
$\ol\rho(\tau,y)-\rho_\infty(y)=\ol n(\tau,y)=0$ for $|y|>M$ and $\tau\in[0,T]$.
Let $\psi\in C^1_c(\R)$ with $\psi=1$ in $[-M,M]$ and 
define $\chi(\tau,y)=\varphi(\tau)\psi(y)$.
We use this test function in \eqref{eq:relen.test1},
\eqref{eq:relen.test2} and \eqref{eq:relen.test3},
and we add the resulting identities and~\eqref{eq:hrhoinf} to inequality~\eqref{eq:weakEnt.total.rhoinf},
which yields
\begin{equation}\label{eq:relentr.weak.proof2}
\begin{aligned}
    &\int_0^\infty -\varphi'\calE(\tau;\rho,n\,\vert\,\ol\rho,\ol n)\dd\tau
    -\varphi(0)\calE(0;\rho_0,n_0\,\vert\,\ol\rho(0),\ol n(0))
    \\
    &\qquad
    \leq \int_0^\infty \varphi\big(
    - \frac{1}{2} \calE(\tau;\rho,n\,\vert\,\ol\rho,\ol n) 
    - \damp(\rho,n\,\vert\,\ol\rho,\ol n)
    +\Xi\big)\dd\tau
    \\
    &\qquad
    +\int_0^\infty\!\!\int_\R \frac{1}{2}\varphi\big(h(\rho_\infty)-h(\ol\rho)-h'(\rho_\infty)\rho_\infty+h'(\ol\rho)\ol\rho\big)  
    \dd y\dd\tau
    \\
    &\qquad
    -\int_0^\infty\!\!\int_\R 
    \varphi\Big(\frac{\ol n}{\ol\rho}\Big)_y\Big(\pp(\ol\rho)+\pp'(\ol\rho)(\rho-\ol\rho)\Big)
    \dd y\dd\tau
    \\
    &\qquad
    +\int_0^\infty\!\!\int_\R 
    \varphi\Big(\rho \Big(\frac{n}{\rho}-\frac{\ol n}{\ol\rho}\Big)\frac{\pp(\ol\rho)_y}{\ol\rho}
    +h''(\ol\rho)\Big[\frac{y}{2}\ol\rho\,\ol\rho_y+(\rho-\ol\rho)\ol n_y\Big]\Big)
    \dd y\dd\tau
    \\
    &\qquad
    -\int_0^\infty\!\!\int_\R 
    \varphi\Big(
    h'(\rho_\infty)_y \frac{y}{2}\rho_\infty
    + h''(\ol\rho) \ol\rho_y n
    \Big)
    \dd y\dd\tau.
\end{aligned}
\end{equation}
Since $\ol\rho-\rho_\infty$ has compact support in space, 
integration by parts implies
\[
\int_\R \frac{1}{2}\big(h(\rho_\infty)-h(\ol\rho)-h'(\rho_\infty)\rho_\infty+h'(\ol\rho)\ol\rho\big)  
\dd y
= \int_\R \frac{y}{2}\big(h'(\rho_\infty)_y\rho_\infty-h(\ol\rho)_y\ol\rho  \big)
\dd y.
\]
Moreover, from~\eqref{eq:h.properties} we conclude
$\pp(\ol\rho)_y=h''(\ol\rho)\ol\rho\,\ol\rho_y$
and
\[
\begin{aligned}
\int_\R\Big(\frac{\ol n}{\ol\rho}\Big)_y\Big(\pp(\ol\rho)+\pp'(\ol\rho)(\rho-\ol\rho)\Big)
+\rho\frac{\ol n}{\ol\rho^2}\pp(\ol\rho)_y
-h''(\ol\rho)(\rho-\ol\rho)\ol n_y\dd y
= \int_\R \Big(\frac{\ol n}{\ol\rho}\pp(\ol\rho)\Big)_y \dd y 
=0
\end{aligned}
\]
since $\ol n$ is compactly supported in space.
Now,~\eqref{eq:relentr.weak.proof2} simplifies 
and we arrive at~\eqref{eq:relent.total.weak}.
\end{proof}

To extend the class of admissible functions $(\ol \rho,\ol n)$,
we first use Lemma~\ref{lem:coercivity} to obtain 
function spaces that contain the weak solutions
satisfying~\ref{item:E2}
if $\pp$ is of the form~\eqref{eq:pres.gamma}.

\begin{proposition}
    \label{prop:weaksol.spaces}
    Let $\pp$ satisfy~\eqref{eq:pres.gamma}, and let 
    either $\rho_+=\rho_-=0$ or $\rho_+,\rho_->0$.
    Let
    $(\rho,n)$ be a weak solution to~\eqref{eq:Euler.trans}
    satisfying~\ref{item:E2}
    with initial data $(\rho_0,n_0)\in L^1_{\mathrm{loc}}(\R)^2$
    such that $\EE(0;\rho_0,n_0\,\vert \,\hat\rho,0)<\infty$.
    Then
    \[
    \begin{aligned}
    \rho-\hat\rho&\in L^\infty(0,\infty; L^2(\R)+L^\gamma(\R)),
    \\
    \ee^{-\tau/2}n &\in L^\infty(0,\infty;L^2(\R)+L^{2\gamma/(\gamma+1)}(\R)).
    \end{aligned}
    \]
\end{proposition}
\begin{proof}
    In virtue of Lemma~\ref{le:relentr.E2.compact.supp},
    we can assume that $\rho_\infty\equiv 0$ or $\inf\rho_\infty>0$.
    Since we have 
    $\EE(\tau;\rho,n\,\vert\,\rho_\infty,0)\in L^\infty(0,\infty)$
    by~\ref{item:E2}, the coercivity estimate from Lemma~\ref{lem:coercivity}
    with $K=\rho_\infty(\R)$
    yields the claim since $\hat\rho-\rho_\infty\in L^2(\R)+L^\gamma(\R)$.    
\end{proof}

We can now extend the result of Lemma~\ref{le:relentr.E2.compact.supp}
to more general pairs $(\ol\rho,\ol n)$.

\begin{theorem}\label{thm:relentr.weak.E2}
	Let $\pp$ satisfy~\eqref{eq:pres.gamma}.
    Let $(\rho,n)$ be a weak solution 
    to \eqref{eq:Euler.trans} fulfilling~\ref{item:E2}
    with initial data $(\rho_0,n_0)\in L^1_{\mathrm{loc}}(\R)^2$
    such that $\EE(0;\rho_0,n_0\,\vert \,\hat\rho,0)<\infty$.
    Let $(\ol\rho,\ol n)\in C^1((0,\infty)\times\R)^2\cap C([0,\infty)\times\R)^2$ be such that
    $\inf\ol\rho>0$ and
    \begin{equation}
    \label{eq:cond.relentr.E2}
    \begin{aligned}
    \ol\rho_y, \ \ol n_y, \ 
    \ol \rho_\tau,\ \ol n_\tau, \ 
    y(\ol\rho-\hat\rho), \ y\,\ol n
    &\in L^\infty_{\mathrm{loc}}([0,\infty);L^1(\R)\cap L^\infty(\R)).
    \end{aligned}
    \end{equation}
    Then the relative entropy inequality~\eqref{eq:relent.total.weak}
    holds.
\end{theorem}

\begin{proof}
Let $\phi \in C_c^\infty([0, \infty))$ be a smooth cut-off function with $\phi \equiv 1 $  on $[0,1]$ and $\phi \equiv 0$ on $[2, \infty)$,
and set $\phi_R(y) := \phi(|y|/R)$ for $R>0$.
We approximate the reference pair $(\ol \rho, \ol n)$ by 
\[
\ol \rho_R := \phi_{2R} \big(\ol \rho {-} \rho_\infty\big) + \rho_\infty
=\phi_{2R} \ol \rho +(1-\phi_{2R})\rho_\infty,
\qquad
\ol n_R := \phi_R \ol n,
\]
so that $\frac{\ol n_R}{\ol \rho_R} = \phi_R \frac{\ol n}{ \ol \rho}$.
By Lemma~\ref{le:relentr.E2.compact.supp}, inequality~\eqref{eq:relent.total.weak} holds 
for $(\ol\rho,\ol n)$ replaced with $(\ol\rho_R,\ol n_R)$.
In order to pass to the limit $R \to \infty$,
we consider the terms in~\eqref{eq:relent.total.weak} separately.
To this end, note that~\eqref{eq:cond.relentr.E2}
implies $\ol\rho-\rho_\infty,\, \frac{y}{2}(\ol\rho-\rho_\infty),\,\ol n,\,\ol R_1,\,\ol R_2\in L^1_{\mathrm{loc}}([0,\infty);L^1(\R)\cap L^\infty(\R))$
and 
$\EE(\tau; \ol \rho, \ol n \, | \, \hat\rho, 0)\in L^\infty_{\mathrm{loc}}([0,\infty))$.
Moreover, this yields
$\EE(\tau; \rho, n \, \vert \, \ol \rho, \ol n )
\in L^\infty_{\mathrm{loc}}([0,\infty))$
and $\EE(0; \rho_0, n_0 \, \vert \, \ol \rho(0), \ol n(0) )
<\infty$
by Proposition~\ref{prop:relent.difference}.

To see the convergence of the relative entropy terms,
we write the kinetic part as
\[
\ee^{-\tau}\frac{\rho}{2}\Big\vert \frac{n}{\rho} -\frac{\ol n_R}{\ol \rho_R} \Big\vert^2
= \ee^{-\tau}\frac{\rho}{2}\Big\vert \frac{n}{\rho} -\frac{\ol n}{\ol \rho} \Big\vert^2
-\ee^{-\tau}(\phi_R-1) n\frac{\ol n}{\ol\rho}
+\ee^{-\tau}(\phi_R^2-1)\frac{\rho}{2}\frac{\ol n^2}{\ol\rho^2},
\]
while the relative potential is subject to
\[
\begin{aligned}
&h(\rho\,\vert\,\ol\rho_R)
=h(\rho\,\vert\,\ol\rho)
\\
&-\int_0^1 h''\big(\ol\rho+\theta(1-\phi_{2R})(\rho_\infty-\ol\rho)\big)
\big(\rho-\ol\rho-\theta(1-\phi_{2R})(\rho_\infty-\ol\rho)\big)
(1-\phi_{2R})(\rho_\infty-\ol\rho)\dd\theta,
\end{aligned}
\]
where we used~\eqref{eq:h.properties}.
Using the integrability properties from Proposition~\ref{prop:weaksol.spaces}
and the assumptions~\eqref{eq:cond.relentr.E2},
we thus conclude
\[
\begin{aligned}
\EE(0; \rho_0, n_0 \, \vert \, \ol \rho_R(0), \ol n_R(0) )
&\to\EE(0; \rho_0, n_0 \, \vert \, \ol \rho(0), \ol n(0) )
&&\quad\text{as }R\to\infty,
\\
\EE(\tau; \rho, n \, \vert \, \ol \rho_R, \ol n_R )
&\to\EE(0; \rho, n \, \vert \, \ol \rho, \ol n )
&\text{in }L^\infty_\mathrm{loc}([0,\infty))
&\quad\text{as }R\to\infty.
\end{aligned}
\]
In the same way, we deduce the convergence of the relative damping term $\damp$.

It thus remains to show the convergence of the terms associated with $\Xi$.
We consider the terms $\ol R_{1,R}$, $\ol R_{2,R}$ and $\xi_R$,
which are defined by replacing $(\ol \rho,\ol n)$ with $(\ol\rho_R,\ol n_R)$ 
in the definitions of $\ol R_1$, $\ol R_2$ and $\xi$
in~\eqref{eq:residuals} and~\eqref{eq:xi.def}.
Then 
\[
\ol R_{1,R}
=\phi_{2R}\ol R_1-\frac{y}{2}\big((\phi_{2R})_y(\ol\rho-\rho_\infty) 
+(1-\phi_{2R})(\rho_\infty)_y\big)
+(\phi_{R}-\phi_{2R})\ol n_y+(\phi_{R})_y\ol n,
\]
so that $\ol R_{1,R}\to\ol R_1$ strongly in $L^\infty_\mathrm{loc}([0,\infty);L^1(\R))$ 
and weakly$^*$ in $L^\infty_\mathrm{loc}([0,\infty);L^\infty(\R))$
as $R\to\infty$
by~\eqref{eq:cond.relentr.E2}.
Similarly, we see
$\ol R_{2,R}\to\ol R_2$ 
as $R\to\infty$
in the same topologies.
Now the estimates from Lemmas~\ref{le:xi.1.bound}, \ref{le:xi.2.bound}, \ref{le:xi.3.bound}
yield an upper bound for $|\xi_R|$
that converges in $L^\infty_\mathrm{loc}([0,\infty);L^1(\R))$ as $R\to\infty$.
Since we also have $\xi_R\to\xi$ pointwise,
we conclude from Pratt's dominated convergence theorem that
$\xi_R\to\xi$ in $L^\infty_\mathrm{loc}([0,\infty);L^1(\R))$ as $R\to\infty$,
which finally yields~\eqref{eq:relent.total.weak}.
\end{proof}

Let us compare the two main results on the relative entropy inequality
from Theorem~\ref{thm:relentr.weak.E1} and Theorem~\ref{thm:relentr.weak.E2}
under the different criteria~\ref{item:E1} and~\ref{item:E2}.
While~\ref{item:E1} is more restrictive than~\ref{item:E2}, 
see Remark~\ref{rem:E1.implies.E2},
the corresponding assumption~\eqref{eq:cond.relentr.E1} on the pair $(\ol\rho,\ol n)$ is slightly more general than~\eqref{eq:cond.relentr.E2}.
Although both assumptions~\eqref{eq:cond.relentr.E1} and~\eqref{eq:cond.relentr.E2} are clearly not optimal,
they are sufficient for our purposes.

\subsection{Weak-strong uniqueness}

The relative entropy inequality~\eqref{eq:relent.total.weak} directly leads to a weak-strong uniqueness result. The proof relies on the following generalized version of Gronwall's lemma in differential form, which enables us to derive a bound on the total relative entropy.
This lemma will also be essential in Section~\ref{se:ConvProof} to prove the main results on exponential convergence with respect to the scaled time $\tau$.

\begin{lemma}[Gronwall inequality]\label{lemma:Gronwall.weak}
    Let $E_0\in\R$ and $E,a,b\in L^1_\mathrm{loc}([0,\infty))$ such that $\tau \mapsto a(\tau)E(\tau) $ belongs to  $ L^1_\mathrm{loc}([0,\infty))$, and
    \[
    \int_0^\infty -\varphi'(\tau) E(\tau) \dd\tau
    \leq E_0\varphi(0)+\int_0^\infty\varphi(\tau) \big( a(\tau) E(\tau) + b(\tau)\big)\dd\tau
    \]
    for all $\varphi\in C^1_c([0,\infty))$ with $\varphi\geq 0$.
    Then
    \[
    E(\tau)\leq E_0 \,\ee^{\int_0^\tau a(\nu)\dd\nu}
    + \int_0^\tau b(\sigma)\,\ee^{\int_\sigma^\tau a(\nu)\dd\nu}\dd\sigma \quad \text{ for a.a.}~\tau >0.
    \]
\end{lemma}
\begin{proof}
We test with $\varphi(\sigma):= \omega_\delta(\sigma) \, \ee^{-\int_0^{\sigma} a(\nu)\dd\nu}$, where we choose $\omega_\delta \in C^1_c([0,\infty))$ satisfying $\omega_\delta(\sigma) \equiv 1$ for $\sigma \leq  \tau$ and $\omega_\delta(\sigma) \equiv 0$ for all $\sigma \geq \tau+\delta$ such that $\vert \omega_\delta'(\sigma) + \frac1\delta \vert \leq \delta $
for all $\sigma \in (\tau,\tau+\delta)$.
Then, we obtain
\begin{align*}
  \int_0^{\tau+\delta} - {\omega_\delta'}(\sigma) \, \ee^{-\int_0^{\sigma} a(\nu)\dd\nu} E(\sigma) \dd\sigma 
    \leq E_0 +\int_0^{\tau+\delta}  \omega_\delta(\sigma)\,b(\sigma) \, \ee^{-\int_0^{\sigma} a(\nu)\dd\nu}   \dd\sigma.
\end{align*}
Taking the limit $\delta \searrow 0$ yields 
\begin{align*}
    \ee^{-\int_0^{\tau} a(\nu)\dd\nu} E(\tau) 
    \leq E_0 + \int_0^\tau b(\sigma) \, \ee^{-\int_0^{\sigma} a(\nu)\dd\nu}  \dd \sigma \quad\text{for a.e.~}\tau>0,
\end{align*}
which is equivalent to the desired result.
\end{proof}

This lemma enables us to derive the following weak-strong uniqueness result from the relative entropy inequality~\eqref{eq:relent.total.weak}.

\begin{corollary}
	\label{cor:ws.uniqueness}
    Let $\pp$ satisfy~\eqref{eq:pres.gamma}.
    Consider initial data $(\rho_0,n_0)\in L^1_{\mathrm{loc}}(\R)^2$ 
such that $\EE(0;\rho_0,n_0\,\vert \,\hat\rho,0)<\infty$.
Let $(\rho, n)$ be a weak solution to~\eqref{eq:Euler.trans} satisfying~\ref{item:E1} or~\ref{item:E2}.
    Let $T>0$ and suppose there exists a strong solution $(\ol \rho, \ol n) \in C^1((0,T)\times\R)^2\cap C([0,T)\times\R)^2$ to~\eqref{eq:Euler.trans}, having the same initial data,     
    such that $\inf\ol\rho>0$ and~\eqref{eq:cond.relentr.E2} holds.
    Then $(\rho,n)$ and $(\ol \rho, \ol n)$ coincide almost everywhere in $[0,T] \times \R$.
\end{corollary}

\begin{proof}
Since~\eqref{eq:cond.relentr.E2} implies~\eqref{eq:cond.relentr.E1},
by Theorem~\ref{thm:relentr.weak.E1}
or Theorem~\ref{thm:relentr.weak.E2} 
we obtain the relative entropy inequality~\eqref{eq:relent.total.weak}.
Since $(\ol \rho, \ol n)$ is a strong solution, the residuals $\ol R_1$ and $\ol R_2$ are identically zero, which implies $\xi_2 \equiv \xi_3 \equiv 0$. 
Further, we use 
Lemma~\ref{le:xi.1.bound} to bound $\xi_1$ by the relative entropy, which
leads to 
   \begin{align*}
   \begin{aligned}
    &\int_0^T -\varphi'\,  \EE(\tau;\rho,n \, \vert \, \ol \rho,\ol n)  
    + \frac12 \varphi \EE(\tau;\rho,n \, \vert \, \ol \rho,\ol n) \dd\tau \\
   &\qquad \leq \int_0^T\!\!\!\int_\R \varphi \,\xi_1 \dd y \dd\tau
     \leq \max\{2,\gamma -1\} \Big\|\Big[ \Big( \frac{\ol n}{\ol \rho}\Big)_y\Big]_+ \Big\|_{L^\infty((0,T)\times\R)} \int_0^T \varphi \EE(\tau;\rho,n \, \vert \, \ol \rho,\ol n) \dd\tau
   \end{aligned}  
   \end{align*}
    for all $\varphi\in C^1_c([0,\infty))$ with $\varphi\geq 0$,
    since $\EE(0; \rho_0, n_0 \, | \, \ol\rho(0),\ol n(0))=0$.
Applying the generalized Gronwall inequality from Lemma~\ref{lemma:Gronwall.weak}, 
we arrive at $\EE(\tau;\rho,n \, \vert \, \ol \rho,\ol n) \leq 0$ for a.a.~$\tau>0$, which yields $(\rho,n) = (\ol \rho, \ol n)$ almost everywhere in $[0,T] \times \R$. 
\end{proof}

\section{Convergence towards similarity profiles} \label{se:ConvProof}

In this section we prove our main results in the one-dimensional case,
namely the convergence of weak solutions 
$(\rho,n)$ to the scaled system \eqref{eq:Euler.trans} towards the similarity profile $(\rho^*,n^*)$ defined by \eqref{eq:profileEq}.
To this end, we consider the residuals $\ol R_i$ as defined in~\eqref{eq:residuals}, and the function $\xi$ given in~\eqref{eq:xi.def}, all with respect to the profile $(\rho^*, n^*)$ and denote the residuals with $R_i^*$.

\subsection{Coincident limits at spatial infinity} \label{su:coincidentLimits}

We first consider the case when the limits $\rho_{\pm}$ coincide,
so that $\hat\rho=\rho_+=\rho_-\geq 0$.
Then, the corresponding similarity profile
is constant and given by
$(\rho^*,n^*)(y) \equiv (\hat \rho,0)$.
In particular, we have $R^*_1=R^*_2=0$ and $\xi=0$.
Utilizing inequality~\eqref{eq:relent.total.weak}
for the total relative entropy, we deduce the convergence result in the case of coincident limits.

\begin{proof}[Proof of Theorem~\ref{thm:convergence.samelimits}]
Depending on whether we assume~\ref{item:E1} or~\ref{item:E2},
we can use Theorem~\ref{thm:relentr.weak.0}
or
Lemma~\ref{le:relentr.E2.compact.supp}, which implies
\[
    \int_0^\infty\!\! -\varphi'\, \EE(\tau;\rho,n \, \vert \, \hat \rho, 0)  
    +  \varphi \Big(\frac12\EE(\tau;\rho,n \, \vert \, \hat \rho, 0) 
    + \damp(\rho,n\,\vert\,\hat\rho,0)\Big)\dd \tau
   \leq \varphi(0)\,\EE(0;\rho_0,n_0 \, \vert \, \hat \rho, 0)
\]
for all $\varphi\in C^1_c([0,\infty))$ with $\varphi\geq 0$.
We omit $\damp(\rho,n\,\vert\,\hat\rho,0)\geq0$ on the left-hand side. 
An application of the generalized Gronwall inequality from Lemma~\ref{lemma:Gronwall.weak} 
yields~\eqref{eq:convergence.samelimits.entropy}.

    In the previous inequality we can also use a sequence of test functions $\varphi$ 
    that suitably approximates the characteristic function of an interval $[\tau,\sigma]$ for $\tau<\sigma$.
    Arguing similarly as in the proof of Lemma~\ref{lemma:Gronwall.weak}, 
    this leads to
    \[
    \EE(\sigma;\rho(\sigma),n(\sigma) \, \vert \, \hat \rho, 0)
    + \int_\tau^\sigma \frac12\EE(\tau';\rho,n \, \vert \, \hat \rho, 0) 
    + \damp(\rho,n\,\vert\,\hat\rho,0)\dd \tau'
   \leq \,\EE(\tau;\rho(\tau),n(\tau) \, \vert \, \hat \rho, 0)
   \]
   for a.a.~$\tau,\sigma\in(0,\infty)$ with $\tau<\sigma$.
   Omitting the entropy terms on the left-hand side,
   and using the decay estimate~\eqref{eq:convergence.samelimits.entropy},
   we obtain~\eqref{eq:convergence.samelimits.dissipation}.
\end{proof}

\subsection{Different limits at spatial infinity}
\label{su:generalCase}

Now we address the case where the limits $\rho_\pm$ are both positive and 
do not coincide, that is, $\rho_\pm>0$ and $\rho_+\neq\rho_-$.
Hence, the unique similarity profile $(\rho^*,n^*)$ solving \eqref{eq:profileEq} 
is not constant.
Further, by using that $\rho^*$ and $n^*$ are time-independent,
the corresponding residuals $R_i^*$ are given by
$R^*_1=0$ and $R^*_2=R^*$
with $R^*$ defined in~\eqref{eq:R.def}.
This in turn leads to simplified expressions for 
$\xi_j$ from~\eqref{eq:xi.def} 
as we now have
    \begin{equation} \label{eq:xi.profile}
        \xi_1 =  - \Big( \frac{n^*}{\rho^*}\Big)_y\big( \ee^{-\tau} \rho \Big\vert \frac n\rho -\frac{ n^*}{\rho^*} \Big\vert^2 + \pp (\rho \, \vert \, \rho^*) \big) , \quad 
        \xi_2 = -  \ee^{-\tau } \frac{R^*}{\rho^*} \Big(\frac n\rho -\frac{n^*}{\rho^*}\Big) \rho, \quad
        \xi_3 = 0.
    \end{equation}
Notice that the exponentially growing term 
$\ee^\tau \big( \pp(\rho^*)_y + \alpha n^*)$ in $R^*_2$ from~\eqref{eq:residuals},
which could counteract convergence as $\tau\to\infty$, 
vanishes due to the profile equation.

For the proof of Theorem~\ref{thm:Conv}, we transfer the estimates for the terms $\xi_1$ and $\xi_2$ that are now considered with respect to the profile $(\rho^*,n^*)$. 

\begin{lemma} \label{le:estimates.Xi.profile}
The functions $\xi_1$ and $\xi_2$, as defined in~\eqref{eq:xi.profile}, satisfy the estimates
\[
   \xi_1 \leq \theta \, \eta (\tau;\rho,n \, \vert \, \rho^*,n^*),
   \qquad
   |\xi_2| \leq  \mu \, \ee^{-\tau/2} \, \eta(\tau; \rho,n \, | \, \rho^*, n^*) +  \ee^{-\tau/2}  
    | R^*|
\]
with $\theta$ and $\mu$ as in Theorem~\ref{thm:Conv}.
\end{lemma}

\begin{proof}
With $(\ol \rho,\ol n) = (\rho^*,n^*)$, both estimates follow directly by applying Lemmas~\ref{le:xi.1.bound} and~\ref{le:xi.2.bound}, respectively. To determine the constant $\theta$, 
we additionally use the identity $\alpha\,n^*=-\pp(\rho^*)_y=-\rho^*h'(\rho^*)_y$, which holds by the profile equations~\eqref{eq:profileEq}.
\end{proof}

We are now ready to prove Theorem \ref{thm:Conv}.

\begin{proof}[Proof of Theorem \ref{thm:Conv}]
First we show 
$\EE(0;\rho_0,n_0\,\vert\,\rho^*,n^*)<\infty$.
Since we have $\rho_\pm>0$,
and
since $\rho^*-\hat\rho$ and 
$n^*=\frac{1}{\alpha}\pp'(\rho^*)\rho^*_y$
decay exponentially fast by Theorem~\ref{thm:profile},
this follows directly
from Proposition~\ref{prop:relent.difference}
and the assumption $\EE(0;\rho_0,n_0\,\vert\,\hat\rho,0)<\infty$.

Depending on whether we assume~\ref{item:E1} or~\ref{item:E2},
we can use Theorem~\ref{thm:relentr.weak.E1}
or
Theorem~\ref{thm:relentr.weak.E2}
with $(\ol\rho,\ol n)=(\rho^*,n^*)$ to obtain
the corresponding relative entropy inequality~\eqref{eq:relent.total.weak}.
Notice that this choice is admissible
due to the decay estimates from Theorem~\ref{thm:profile}.
Estimating $\xi$ with the help of Lemma~\ref{le:estimates.Xi.profile},
we arrive at 
   \begin{equation}
   \label{eq:entr.ineq.final}
   \begin{aligned}
    &\int_0^\infty\!\! -\varphi'\, \EE(\tau;\rho,n \, \vert \, \rho^*,n^*)  
    +  \varphi \Big( \frac12 \EE(\tau;\rho,n \, \vert \, \rho^*, n^*) + \damp(\rho,n \, \vert \, \rho^*,n^*) \Big)\dd\tau
    \\
     &\quad
     \leq \varphi(0)\,\EE(0;\rho_0,n_0 \, \vert \, \rho^*,n^*)
     +\int_0^\infty\!\!\varphi \Big( (\theta+\mu\ee^{-\tau/2})\EE(\tau;\rho,n \, \vert \, \rho^*, n^*)
     + K \ee^{-\tau/2}\Big)\dd \tau.
   \end{aligned}  
   \end{equation}
We first omit the term $\damp(\rho,n \, \vert \, \rho^*,n^*)\geq 0$. 
Then the generalized Gronwall inequality from Lemma~\ref{lemma:Gronwall.weak} with $a(\tau):= -\frac12 + \theta +  \mu\ee^{-\tau/2}$ and $b(\tau) := K\ee^{-\tau/2} $ yields
\begin{align*}
\EE(\tau;\rho(\tau),n(\tau)  \, \vert \, \rho^*, n^*) 
&\leq \EE(0;\rho_0,n_0 \, \vert \, \rho^*,n^*)\, \ee^{\int_{0}^\tau a(\nu) \dd \nu} 
+ \int_{0}^\tau b(\sigma) \, \ee^{\int_{\sigma}^\tau a(\nu) \dd \nu} \dd \sigma 
\\
&\leq \ee^{-(\frac12{-}\theta)\tau + \frac12 \mu}  \Big( \EE(0;\rho_0,n_0 \, \vert \, \rho^*,n^*)
+ K  \int_{0}^\tau \, \ee^{-\theta \sigma } \dd \sigma \Big) 
\\
&\leq \ee^{-(\frac12{-}\theta)\tau + \frac12 \mu}   \Big( \EE(0;\rho_0,n_0 \, \vert \, \rho^*,n^*)
+ \frac{K}{\theta} \Big), 
\end{align*}
where we used the rough estimate
\begin{align*}
\int_{\sigma}^\tau a(\nu) \dd \nu 
=  - \Big(\frac12-\theta \Big) (\tau {-}\sigma) + \frac{\mu}{2} ( \ee^{-\sigma/2} {-} \ee^{-\tau/2}) \leq   -  \Big(\frac12-\theta \Big)  (\tau {-}\sigma) + \frac{\mu}{2}
\end{align*}
for $0 \leq \sigma \leq \tau$.
This shows~\eqref{eq:convergence}.

Now assume $\theta<\frac{1}{2}$.
Arguing as in the proof of Theorem~\ref{thm:convergence.samelimits}, 
from~\eqref{eq:entr.ineq.final} we obtain
   \[
   \begin{aligned}
    &\EE(\sigma;\rho(\sigma),n(\sigma) \, \vert \, \rho^*,n^*)  
    +  \int_\tau^\sigma \Big( \frac12 \EE(\tau';\rho,n \, \vert \, \rho^*, n^*) + \damp(\rho,n \, \vert \, \rho^*,n^*) \Big)\dd\tau'
    \\
     &\quad
     \leq 
     \EE(\tau;\rho(\tau),n(\tau) \, \vert \, \rho^*,n^*)  
     +\int_\tau^\sigma\!\! \Big( (\theta+\mu\ee^{-\tau'/2})\EE(\tau';\rho,n \, \vert \, \rho^*, n^*)
     + K \ee^{-\tau'/2}\Big)\dd \tau'.
   \end{aligned}  
   \]
We omit the first term and 
choose $\tau$ so large that $\theta+\mu\ee^{-\tau/2}\leq\frac{1}{2}$.
Then we arrive at
   \[
   \begin{aligned}
    \int_\tau^\sigma \damp(\rho,n \, \vert \, \rho^*,n^*) \dd\tau'
    &\leq 
     \EE(\tau;\rho(\tau),n(\tau) \, \vert \, \rho^*,n^*)  
     +\int_\tau^\infty K \ee^{-\tau'/2}\dd \tau'
     \\
     &\leq \ee^{-(\frac12{-}\theta)\tau + \frac12 \mu} \Big( \EE(0;\rho_0,n_0 \, \vert \, \rho^*,n^*) + \frac{K}{\theta} \Big) 
     +2K\ee^{-\tau/2}
   \end{aligned}  
   \]
   for a.a.~$\sigma>\tau\geq 2\log(\frac{2\mu}{1-2\theta})$,
   where we employed~\eqref{eq:convergence}.
   This yields~\eqref{eq:convergence.dissipation} and completes the proof.
\end{proof}

\section{Extension to multiple dimensions}
\label{sec:envar}

In this section we adapt the presented method 
for investigating the long-time behavior of solutions
to the multi-dimensional Euler equations,
given by the system
\begin{equation} \label{eq:Euler.original.md}
\begin{aligned}
\wt \rho_t +  \dv\wt m &= 0, \\
\wt m_t + \dv\Big(\frac{\wt m\otimes \wt m}{\wt\rho} 
+ \pp(\wt \rho)\bbI \Big) + \alpha \wt m
&=0
\end{aligned}
\end{equation}
in $(0,\infty)\times\R^d$, $d\in\N$,
where $\wt\rho\colon(0,\infty)\times\R^d\to[0,\infty)$ and $\wt m\colon(0,\infty)\times\R^d\to\R^d$ are (multi-dimensional) density and momentum fields.
Here, $\bbI\in\R^{d\times d}$ denotes the identity matrix.
Denoting by $\partial_j$ the partial derivatives in space,
we define the divergence of a vector field $v=(v_j)$ and a tensor field $A=(A_{ij})$ by
$\dv v = \partial_j v_j$ and
$(\dv A)_i=\partial_j A_{ij}$, respectively,
where we use Einstein summation convention and implicitly sum over repeated indices.

Instead of two different limits $\rho_\pm$, 
we now prescribe an angular function
\begin{equation}
\label{eq:rho.limit.angular}
\hat\rho\in C^1(\R^d\setminus\{0\};(0,\infty)),
\qquad
\forall\lambda>0\ \forall x\in\R^d\setminus\{0\}\colon \ 
\hat\rho(\lambda x)=\hat\rho(x),
\end{equation}
which is a multi-dimensional generalization of $\hat\rho$ from~\eqref{eq:rho.hat}.
We study solutions with initial values
$(\wt \rho_0,\wt m_0)$ satisfying
\[
(\wt \rho_0-\hat\rho,\wt m_0)(x) \to (0,0)
\quad \text{ as } |x| \to + \infty
\]
in a certain sense.
Since the existence of weak solutions is not known 
in the multi-dimensional case, 
we base our examination on energy-variational solutions.

\subsection{The notion of energy-variational solutions}
We begin with the definition of energy-variational solutions to the damped Euler equations~\eqref{eq:Euler.original.md} in the original coordinates.
To this end, we need to quantify a suitable entropy function.
In original variables, the relative entropy is given by
\[
\wt \EE(\wt\rho,\wt m\,\vert\,\ol\rho,\ol n)
=\int_{\R^d} \wt\eta(\wt\rho,\wt m\,\vert\,\ol\rho,\ol m)\dd x,
\qquad
\wt\eta(\wt\rho,  \wt m \,\vert\,\ol\rho,\ol m)
=\frac{1}{2}\wt\rho\Big|\frac{ \wt m }{\wt\rho}-\frac{\ol m}{\ol \rho}\Big|^2 + h(\wt\rho\,\vert\,\ol\rho).
\]
Similarly to before, we consider special reference densities
$\wt\rho_\infty\in C^1([0,\infty)\times\R^d)$ such that
\begin{equation}
    \label{eq:wtrhoinfty}
    \forall\, t\in [0,\infty) \ \exists\, R>0 : \quad 
    \wt\rho_\infty(t,x)=\hat\rho(x) \quad\text{for }|x|>R.
\end{equation}

\begin{definition} \label{def:envarsol.unscaled}
    Let $(\wt \rho_0,\wt m_0)\in L^1_\mathrm{loc}(\R^d)\times L^1_\mathrm{loc}(\R^d;\R^d)$ such that 
    $\wt\EE(\wt \rho_0,\wt m_0\,\vert\,\hat\rho,0)<\infty$.
    We call $(\wt\rho, \wt m,\wt E)\in L^1_{\mathrm{loc}}([0,\infty)\times \R^d)\times L^1_{\mathrm{loc}}([0,\infty)\times \R^d;\R^d)\times \BV([0,\infty))$ 
    an \textit{energy-variational solution} to \eqref{eq:Euler.original.md}
    if there is $\wt\rho_\infty\in C^1([0,\infty)\times\R^d)$ with \eqref{eq:wtrhoinfty}
    such that
    $\wt\EE(\wt \rho,\wt m\,\vert\,\wt\rho_\infty,0)\leq \wt E$ a.e.~in $(0,\infty)$, 
    and if there exists 
    a $1$-homogeneous \textit{regularity weight}
    $\wt{\mathcal K}\colon C_c^1(\R^d;\R^d)\to[0,\infty)$ such that
    \begin{equation}
     \label{eq:envar.unscaled}
    \begin{aligned}
       &\Big[\wt E+\int_{\R^d} \phi\,\wt \rho+\psi\cdot\wt m \dd x\Big]\Big|_s^t
       +\int_s^t\!\!\!\int_{\R^d} \alpha \frac{|\wt m|^2}{\wt\rho} \dd x\dd t' 
       \\
       &\qquad
       -\int_s^t\!\!\!\int_{\R^d} \phi_t \,  \wt \rho + \nabla\phi\cdot\wt m 
         + \psi_t \cdot \wt m + \nabla\psi : \Big( \frac{ \wt m\otimes\wt m}{\wt \rho} + \pp(\wt \rho)\bbI \Big) -  \alpha\psi \cdot\wt m  \dd x \dd t'
        \\
        &\qquad\qquad
        \leq 
        \int_s^t\!\!\!\int_{\R^d} -h'(\wt\rho_\infty)_t(\wt\rho-\wt\rho_\infty) - \nabla h'(\wt\rho_\infty)\cdot\wt m \dd x\dd t'
        \\
        &\qquad\qquad\qquad
        +\int_s^t \wt{\mathcal K}(\psi) \big(\wt E-\wt\EE(\wt \rho,\wt m\,\vert\,\wt\rho_\infty,0)\big)\dd t'
    \end{aligned}
    \end{equation}
   for all $\phi \in C_c^1([0,\infty) \times \R^d)$,
   $\psi \in C_c^1([0,\infty) \times \R^d;\R^d)$
   and for a.a.~$s,t\in[0,\infty)$ with $s<t$, 
   including $s=0$ with $(\wt\rho,\wt m)(0)=(\wt\rho_0,\wt m_0)$.
\end{definition}

As before, with the assumption
$\wt\EE(\wt \rho_0,\wt m_0\,\vert\,\hat\rho,0)<\infty$
we implicitly assume $\rho_0\geq 0$ a.e.
Furthermore, the statement and proof of Proposition~\ref{prop:relent.difference}
can immediately be transferred to the multi-dimensional setting. Therefore, $\wt\EE(\wt \rho_0,\wt m_0\,\vert\,\hat\rho,0)<\infty$
is equivalent to $\wt\EE(\wt \rho_0,\wt m_0\,\vert\,\wt\rho_\infty(0),0)<\infty$.

\begin{remark}\label{rem:envar.unscaled}
    Since $\mathcal K$ is 1-homogeneous,
    that is, $\wt{\mathcal K}(\lambda\psi)=\lambda\wt{\mathcal K}(\psi)$
    for $\lambda>0$, 
    the formulation~\eqref{eq:envar.unscaled} is equivalent to
    the accumulated formulas
    \begin{subequations}
    \label{eq:envar.unscaled.split}
    \begin{align}
     \label{eq:envar.unscaled.dens}
    \int_{\R^d}\phi\,\wt \rho \dd x\Big|_s^t
       -\int_s^t\!\!\!\int_{\R^d} \phi_t \,  \wt \rho &+ \nabla\phi\cdot \wt m 
        \dd x \dd t'
        = 0,
    \\
     \label{eq:envar.unscaled.mom}
    \begin{split}
    \int_{\R^d} \psi\cdot\wt m \dd x\Big|_s^t
    -\int_s^t\!\!\!\int_{\R^d}  \psi_t \cdot \wt m 
    &+ \nabla\psi:\Big( \frac{ \wt m\otimes\wt m}{\wt \rho} + \pp(\wt \rho)\bbI \Big) -  \alpha\psi \cdot\wt m  \dd x \dd t' 
    \\
    &\leq
    \int_s^t \wt{\mathcal K}(\psi) \big(\wt E -\wt\EE(\wt \rho,\wt m\,\vert\,\wt\rho_\infty,0)\big)\dd t',
    \end{split}
    \\
     \label{eq:envar.unscaled.enin}
    \wt E\big|_s^t
       +\int_s^t\!\!\!\int_{\R^d} \alpha \frac{|\wt m|^2}{\wt\rho} \dd x\dd t' 
       &\leq 
        \int_s^t\!\!\!\int_{\R^d} -h'(\wt\rho_\infty)_t(\wt\rho-\wt\rho_\infty) - \nabla h'(\wt\rho_\infty)\cdot\wt m \dd x\dd t',
    \end{align}
    \end{subequations}
    for all $\phi,\psi$ and $s<t$ as above.
    This can be seen by a classical variational argument: 
    Choose test functions $\mu\phi$ and $\lambda\psi$
    for $\mu,\lambda\geq 0$
    and consider the limit cases $\mu=0$, $\lambda=0$, $\mu\to\infty$ and $\lambda\to\infty$.
    While~\eqref{eq:envar.unscaled.dens} corresponds to a standard weak formulation of the continuity equation, the main difference between weak solutions 
    and energy-variational solutions to~\eqref{eq:Euler.original.md} is visible in~\eqref{eq:envar.unscaled.mom}
    and~\eqref{eq:envar.unscaled.enin}.
    The energy variable $\wt E$ occurs in both equations as an additional auxiliary variable.
    Inequality~\eqref{eq:envar.unscaled.mom} relaxes the momentum equation to an inequality,
    where the right-hand side measures the deviation from a weak form
    in terms of the \textit{energy defect} $\wt E-\wt\EE(\wt \rho,\wt m\,\vert\,\wt\rho_\infty,0)\geq 0$
    multiplied by the \textit{regularity weight} $\wt{\mathcal K}$.
    Moreover,~\eqref{eq:envar.unscaled.enin} corresponds to the (relative) entropy inequality, where the mechanical energy is replaced with $\wt E$.
    In particular, if 
    $\wt E = \wt\EE(\wt \rho,\wt m\,\vert\,\wt\rho_\infty,0)$,
    then \eqref{eq:envar.unscaled.split} reduces to a 
    classical weak formulation of~\eqref{eq:Euler.original.md} combined with an energy inequality.
\end{remark}

\begin{remark}
    In Definition~\ref{def:envarsol.unscaled}
    the regularity weight $\wt{\mathcal K}$ is not specified further. 
    Usually, it is chosen in such a way that the inequality~\eqref{eq:envar.unscaled} 
    is preserved under weak convergence or such that the solution set is convex,
    which ensures the existence of energy-variational solutions or
    properties like weak-strong uniqueness.
    In the case of the barotropic pressure law~\eqref{eq:pres.gamma},
    a suitable choice is given by
    \begin{equation}\label{eq:K}
    \wt{\mathcal K}(\psi):=\max\big\{2\|[\nabla\psi]_\mathrm{sym,+}\|_{L^\infty(\R^d)},(\gamma-1)\|[\dv\psi]_+\|_{L^\infty(\R^d)}
    \big\},
    \end{equation}
    see \cite{EitLas24envarhyp},
    where the system was studied in a torus domain. 
    Here, $[A]_{\mathrm{sym},+}$
    denotes the positive definite part of the symmetric part $\frac{1}{2}(A+A^\top)$ of a matrix $A\in\R^{d\times d}$,
    while $[a]_+=\max\{a,0\}$ denotes the nonnegative part of a scalar $a\in\R$.
\end{remark}

\begin{remark}
    \label{rem:envar.refdens}
    In order to obtain a finite (relative) energy 
    $\wt\EE(\wt\rho,\wt m\,\vert\,\wt\rho_\infty,0)$,
    the existence of a specific reference density $\wt\rho_\infty$
    with suitable values as $|x|\to\infty$ is necessary.
    However, the state variable $(\wt\rho,\wt m)$ 
    of an energy-variational solution $(\wt\rho,\wt m,\wt E)$
    is independent of this choice.
    Indeed, if we have two reference densities
    $\rho^1_\infty$ and $\rho^2_\infty$,
    then $(\wt\rho,\wt m, E^1)$ is an energy-variational solution to~\eqref{eq:Euler.original.md}
    for $\wt\rho_\infty=\rho^1_\infty$
    if and only if $(\wt\rho,\wt m, E^2)$ is an energy-variational solution
    for $\wt\rho_\infty=\rho_\infty^2$,
    where 
    \[
    E^2= E^1+\int_{\R^d} h(\rho^1_\infty\,\vert\,\rho^2_\infty)+\big(h'(\rho^1_\infty)- h'(\rho^2_\infty) \big)(\wt\rho-\rho^1_\infty)\dd x.
    \]
    This equivalence follows from the elementary identity
    \[
    \wt\EE(\wt\rho,\wt m\,\vert\,\rho^2_\infty,0)
    =\wt\EE(\wt \rho,\wt m\,\vert\,\rho^1_\infty,0)
    +\int_{\R^d} h(\rho^1_\infty\,\vert\,\rho^2_\infty)+\big(h'(\rho^1_\infty)-h'(\rho^2_\infty)\big)(\wt\rho-\rho^1_\infty)\dd x, 
    \]
    and by substituting the test function $\phi$ with $\phi+h'(\rho^1_\infty)-h'(\rho^2_\infty)$.
\end{remark}

 For the sake of completeness, let us show that, 
in contrast to weak solutions, 
the global-in-time existence of energy-variational solutions to~\eqref{eq:Euler.original.md}
can be shown in the multi-dimensional framework.
To this end,
we use the general existence result from~\cite{ALR24envarvisc},
which allows for the presence of dissipative terms.
Here we denote
the class of weakly continuous functions from $[0,\infty)$ to a Banach space $X$
by $C_w([0,\infty);X)$.

\begin{theorem}
    \label{thm:envar.existence}
    Let $\alpha\geq 0$ and $\pp(\rho)=k\rho^\gamma$ for $k>0$, $\gamma> 1$.
    Let either $\hat\rho$ be as in~\eqref{eq:rho.limit.angular}
    or $\hat\rho\equiv0$.
    Let $(\wt\rho_0,\wt m_0)\in L^1_{\mathrm{loc}}(\R^d)\times L^1_{\mathrm{loc}}(\R^d;\R^d)$ such that
    $\wt\EE(\wt \rho_0,\wt m_0\,\vert\,\hat\rho,0)<\infty$.
    Then there exists an energy-variational solution $(\wt\rho,\wt m,\wt E)$ 
    to~\eqref{eq:Euler.original.md}
    with 
    \[
    \begin{aligned}
    \wt\rho-\hat\rho
    &\in C_w([0,\infty),L^\gamma(\R^d)+L^2(\R^d)),
    \\
    \wt m
    &\in C_w([0,\infty),L^{\frac{2\gamma}{\gamma+1}}(\R^d;\R^d)+L^2(\R^d;\R^d))
    \end{aligned}
    \]
    with $\wt E(0)=\wt\EE(\wt \rho_0,\wt m_0\,\vert\,\hat\rho,0)$.
    The corresponding regularity weight $\wt{\mathcal K}$ is given by~\eqref{eq:K}.
\end{theorem}

\begin{proof}
    We choose a time-independent function $\wt\rho_\infty\in C^1(\R^d)$ satisfying~\eqref{eq:wtrhoinfty} 
    such that $\wt\rho_\infty(\R^d)=\hat\rho(\R^d)$.
    By Remark~\ref{rem:envar.refdens}, this is no restriction.
    By $\wt\EE(\wt \rho_0,\wt m_0\,\vert\,\hat\rho,0)<\infty$
    and the multi-dimensional analog of Proposition~\ref{prop:relent.difference},
    we then have 
    $\wt\EE(\wt \rho_0,\wt m_0\,\vert\,\wt\rho_\infty,0)<\infty$.

    We express an energy-variational solution $(\wt\rho,\wt m)$ 
    in the form $(\wt\sigma+\wt\rho_\infty,\wt m)$, that is, $\wt\sigma=\wt\rho-\wt\rho_\infty$,
    where $(\wt\sigma,\wt m)\in \bbV$ for the reflexive Banach space $\bbV=\bbV_1\times\bbV_2$
    with
    \[
    \bbV_1=L^\gamma(\R^d)+L^2(\R^d),
    \qquad
    \bbV_2=L^{2\gamma/(\gamma+1)}(\R^d;\R^d)+L^2(\R^d;\R^d).
    \]
    We extend $\eta$ to a function on $\R\times\R^d$ by setting
    $\eta(\rho,m)=+\infty$
    if $\rho<0$, and if $\rho=0$ and $m\neq 0$,
    and we define the total (relative) entropy
    \[
    \calF\colon \bbV\to[0,\infty],
    \qquad
    \mathcal F(\wt\sigma,\wt m):=\wt \EE(\wt \sigma+\wt\rho_\infty,\wt m\,\vert\,\wt\rho_\infty,0).
    \]
    Since $\wt\rho_\infty$ is time-independent, the dissipation functional is given by
    \[
    \Psi\colon\bbV\to[0,\infty],
    \qquad
    \Psi(\wt\sigma,\wt m)
    =\int_{\R^d} \frac{\alpha |\wt m|^2}{\wt\sigma+\wt\rho_\infty}  
    +\nabla h'(\wt\rho_\infty)\cdot\wt m \dd x,
    \]
    where we set $\frac{|m|^2}{\rho}=+\infty$
    if $\rho<0$, and if $\rho=0$ and $m\neq 0$.
    We choose the regularity weight $\wt{\mathcal K}$ as in~\eqref{eq:K},
    and show that all assumptions in the abstract existence result
    from~\cite[Theorem 3.3]{ALR24envarvisc} are satisfied.
    In particular,    
    since $\eta$ is strictly convex and lower semicontinuous, 
    this also applies to the associated relative entropy,
    so that $\calF$  
    is strictly convex and lower semicontinuous.
    To show coercivity of $\calF$,
    let $K=\wt\rho_\infty(\R^d)$.
    Then Lemma~\ref{lem:coercivity} yields a constant $r_0\geq 0$ such that~\eqref{eq:eta.coercive} holds
    for all $\ol\rho=\wt\rho_\infty(x)$, $x\in\R^d$.
    For $(\wt\sigma,\wt m)\in\bbV$,
    we introduce the sets
    \[
    M_1=\set{x\in\R^d}{|\wt\sigma(x)+\wt\rho_\infty(x)|\leq r_0},
    \qquad
    M_2=\set{x\in\R^d}{|\wt\sigma(x)+\wt\rho_\infty(x)|> r_0}
    \]
    and the decomposition
    $\wt\sigma=\wt\sigma_1+\wt\sigma_2$, 
    $\wt m=\wt m_1 + \wt m_2$
    with
    $\wt\sigma_j=\chi_{M_j}\wt\sigma$, $\wt m_j=\chi_{M_j}\wt m$ for $j=1,2$. 
    Here, $\chi_M$ denotes the characteristic function of a set $M\subset\R^d$. 
    We then obtain
    \[
    \begin{aligned}
    \mathcal F(\wt\sigma,\wt m)
    &=\int_{M_1}\wt\eta(\wt\sigma+\wt\rho_\infty,\wt m\,\vert\,\wt \rho_\infty,0)\dd x
    + \int_{M_2}\wt\eta(\wt\sigma+\wt\rho_\infty,\wt m\,\vert\,\wt\rho_\infty,0)\dd x
    \\
    &\geq C
    \int_{\R^d}|\wt m_1|^2+|\wt \sigma_1|^2 + |\wt m_2|^\frac{2\gamma}{\gamma+1}+|\wt \sigma_2|^\gamma\dd x
    \geq
    C \|(\wt\sigma,\wt m)\|_{\bbV}^\alpha
    \end{aligned}
    \]
    for some exponent $\alpha>1$ if $\|(\wt\sigma,\wt m)\|_{\bbV}>1$.
    This implies coercivity of $\mathcal F$.
    Choosing $\wt{\mathcal K}$ as in~\eqref{eq:K},
    we further deduce the convexity 
    and lower semicontinuity of
    \[
    (\wt\sigma,\wt m)\mapsto
    \Psi(\wt \sigma,\wt m)
    -\int_{\R^d} \nabla \phi\cdot \wt m + \nabla\psi : \Big( \frac{ \wt m\otimes\wt m}{\wt\sigma+\wt\rho_\infty} + \pp(\wt\sigma+\wt\rho_\infty)\bbI \Big) -  \alpha\psi\cdot \wt m  \dd x 
    +\wt{\mathcal K}(\psi) \wt\calF(\wt \sigma,\wt m)
    \]
    on the convex set $\dom \Psi=\dom \calF$ 
    for all $\phi,\psi\in C_c^1(\R^d)\times C_c^1(\R^d;\R^d)$.
    This is clear for all terms linear in $(\wt\sigma,\wt m)$ 
    as well as for $\Psi$.
    For the remaining terms, we can
    argue
    as in~\cite[Lemma 5.6]{EitLas24envarhyp}.
    The remaining assumptions necessary 
    to apply~\cite[Theorem 3.3]{ALR24envarvisc}
    can be verified as in~\cite[Section 5.2]{EitLas24envarhyp},
    which yields the existence of an energy-variational solution to~\eqref{eq:Euler.original.md}.
\end{proof}

As in the study of the asymptotic behavior of weak solutions
in the previous sections,
in Theorem~\ref{thm:envar.existence}
we exclude reference densities that have zeros but do not vanish identically.
This is due to the coercivity estimate~\eqref{eq:eta.coercive},
which is not valid in this case.

\subsection{Reformulation in parabolic scaling variables}
As in the one-dimensional case, to study the long-time behavior of solutions to~\eqref{eq:Euler.original.md}, 
we consider the problem in parabolic scaling variables.
We define 
$\tau := \log(1{+}t)$ and $ y:= \frac{x}{\sqrt{1{+}t}}$
and set
$\rho(\tau,y) := \wt \rho(t,x)$ and $n(\tau,y) := \ee^{d\tau/2} m(\tau,y) = (1{+}t)^{d/2}\, \wt m(t,x)$.
Observe that in the multi-dimensional framework,
we scale the magnitude of
$\wt m$ by $\ee^{d\tau/2}$ instead of $\ee^{\tau/2}$.
Then~\eqref{eq:Euler.original.md} is equivalent to
\begin{equation} \label{eq:Euler.trans.md}
\begin{aligned}
\rho_\tau &= \frac{y}{2}\cdot\nabla \rho - \dv n, \\
n_\tau &= \frac{y}{2} \cdot\nabla n + \frac d2 n - \dv\Big( \frac{n\otimes n}{\rho} \Big) 
- \ee^{d\tau} \big( \nabla \pp(\rho)_y + \alpha n\big).
\end{aligned}
\end{equation}
The the associated relative entropy is given by
\[
\EE(\tau;\rho,n\,\vert \,\ol\rho,\ol n)
=\int_{\R^d} \eta\big(\tau;\rho(y),n(y)\,\vert\,\ol\rho(y),\ol n(y)\big)\dd y
\]
with
\[
\eta(\tau;\rho,n\,\vert\,\ol\rho,\ol n)
=\frac{1}{2}\ee^{-d\tau}\rho\Big|\frac{n}{\rho}-\frac{\ol n}{\ol \rho}\Big|^2 + h(\rho\,\vert\,\ol\rho).
\]
For the sake of generality, we consider a $\tau$-dependent 
reference density $\rho_\infty\in C^1([0,\infty)\times\R^d)$ such that
\begin{equation}
    \label{eq:rhoinfty.md}
    \forall\, \tau\in [0,\infty) \ \exists\, R>0 : \quad 
    \rho_\infty(\tau,y)=\hat\rho(y) \text{ for }|y|>R,
\end{equation}
where $\hat\rho$ is the given angular limit density satisfying~\eqref{eq:rho.limit.angular}.
Energy-variational solutions to~\eqref{eq:Euler.trans.md} 
are defined as follows.

\begin{definition} \label{def:envarsol}
    Let $(\rho_0,n_0)\in L^1_\mathrm{loc}(\R^d) \times L^1_\mathrm{loc}(\R^d ; \R^d)$ such that 
    $\EE(\rho_0,n_0\,\vert\,\hat\rho,0)<\infty$.
    We call a pair $(\rho, n,E)\in L^1_{\mathrm{loc}}([0,\infty)\times \R^d  )\times L^1_{\mathrm{loc}}([0,\infty)\times  \R^d; \R^d )\times \BV([0,\infty))$ 
    an \textit{energy-variational solution} to \eqref{eq:Euler.trans.md} 
    if there is $\rho_\infty\in C^1([0,\infty)\times\R^d)$ with~\eqref{eq:rhoinfty.md} 
    such that $\EE(\tau;\rho,n\,\vert\,\rho_\infty,0)\leq E$ a.e.~in $(0,\infty)$, and if there exists 
    a $1$-homogeneous \textit{regularity weight}
    $\mathcal K\colon C_c^1(\R^d;\R^d)\to[0,\infty)$ such that
    \begin{equation}
     \label{eq:envar}
    \begin{aligned}
       &\Big[E+\int_{\R^d} \phi\, \rho+\psi\cdot n \dd y\Big]\Big|_\sigma^\tau
       +\int_\sigma^\tau \frac{d}{2}E\dd\tau' 
       +\int_\sigma^\tau\!\!\!\int_{\R^d} \alpha \frac{|n|^2}{\rho} \dd y\dd\tau'
       \\
       &\quad
       -\int_\sigma^\tau\!\!\!\int_{\R^d}  
        \phi_{\tau} \,  \rho - \nabla\phi \cdot \big(\frac y2 \rho - n \big) - \frac d2 \phi \rho
        \dd y\dd \tau' 
       \\
       &\qquad
       -\int_\sigma^\tau\!\!\!\int_{\R^d} 
         \psi_{\tau} \cdot n + \nabla\psi : \Big(- n\otimes \frac y2 + \frac{ n\otimes n}{\rho} + \ee^{d\tau'} \pp(\rho)\bbI \Big) -  \psi \cdot \ee^{d\tau'} \alpha n  \dd y\dd \tau'
        \\
        &\qquad\quad
        \leq 
        \int_\sigma^\tau\!\!\!\int_{\R^d}
        -\Big(h'(\rho_\infty)_\tau -\frac{y}{2}\cdot\nabla h'(\rho_\infty)\Big) (\rho-\rho_\infty) 
        -\nabla h'(\rho_\infty)\cdot n \dd y\dd \tau'
        \\
        &\qquad\qquad
        +\int_\sigma^\tau \ee^{d\tau'}\mathcal K(\psi) \big( E-\EE(\tau;\rho,n\,\vert\,\rho_\infty,0)\big)\dd \tau'
    \end{aligned}
    \end{equation}
    for all $\phi \in C_c^1([0,\infty) \times \R^d)$,
    $\psi \in C_c^1([0,\infty) \times \R^d;\R^d)$
    and for a.a.~$\sigma,\tau\in[0,\infty)$ with $\sigma<\tau$, 
    including $\sigma=0$ with $(\rho,n)(0)=(\rho_0,n_0)$.
\end{definition}

\begin{remark}
Arguing as in Remark~\ref{rem:envar.unscaled}, we see that 
\eqref{eq:envar} corresponds to a relaxed version of 
a weak formulation of~\eqref{eq:Euler.trans.md}
combined with a (relative) entropy inequality. 
In particular, in dimension $d=1$, 
an energy-variational solution $(\rho,n,E)$ 
with $E=\EE(\tau;\rho,n\,\vert\,\rho_\infty,0)$
is the same as a weak solution to~\eqref{eq:Euler.trans} in the sense of Definition~\ref{def:weakEntSol.trans}
that satisfies the entropy inequality from~\ref{item:E2}
if we choose a $\tau$-independent reference density $\rho_\infty$.
Actually, this choice is no restriction, compare Remark~\ref{rem:envar.refdens}.
\end{remark}

We next show that energy-variational solutions to~\eqref{eq:Euler.original.md} and
to~\eqref{eq:Euler.trans.md} are related to each other by parabolic scaling.

\begin{proposition}
    \label{prop:envar.equiv}
    Let $(\wt\rho, \wt m,\wt E)
    \in L^1_{\mathrm{loc}}([0,\infty)\times \R^d)\times L^1_{\mathrm{loc}}([0,\infty)\times \R^d;\R^d)\times \BV([0,\infty))$,
    and set 
    $\rho(\tau,y)=\wt\rho(t,x)$, $n(\tau,y)=\ee^{d\tau/2}\wt m(t,x)$ and $E(\tau)=\ee^{-d\tau/2}\wt E(t)$,
    where $t=\ee^{\tau}-1$ and $x=\ee^{\tau/2}y$.
    Then $(\wt\rho, \wt m,\wt E)$ is an energy-variational solution to~\eqref{eq:Euler.original.md}
    if and only if $(\rho, n, E)$ is an energy-variational solution to~\eqref{eq:Euler.trans.md}
    with the same initial values 
    $(\rho_0,n_0)=(\wt \rho_0,\wt m_0)\in L^1_\mathrm{loc}(\R^d)\times L^1_\mathrm{loc}(\R^d;\R^d)$
    and
    the same regularity weight $\mathcal K=\wt{\mathcal K}$.
    Moreover, 
    the reference densities are related by 
    $\wt\rho_\infty(t,x)=\rho_\infty(\tau,y)$.  
\end{proposition}

\begin{proof}
    We only show one implication. The converse follows along the same lines.
    
    Let $(\wt\rho, \wt m,\wt E)$ be an energy-variational solution to~\eqref{eq:Euler.original.md}.
    Consider
    $\phi \in C_c^1([0,\infty) \times \R^d)$,
    $\psi \in C_c^1([0,\infty) \times \R^d;\R^d)$
    and define 
    $\wt\phi(t,x)=\ee^{-d\tau/2}\phi(\tau,y)$,
    $\wt\psi(t,x)=\psi(\tau,y)$.
    We use $\wt\phi$ and $\wt\psi$ as test functions in~\eqref{eq:envar.unscaled.split}.
    A variable transform shows that
    the scaled quantities $(\rho, n, E)$ satisfy
    \begin{align}
     \label{eq:envar.dens}
    \int_{\R^d} \phi\,\rho \dd y\Big|_\sigma^\tau
       &-\int_\sigma^\tau\!\!\!\int_{\R^d}\phi_{\tau} \,  \rho 
       - \nabla\phi\cdot \big(\frac y2 \rho - n \big) 
       - \frac d2 \phi \rho
        \dd y\dd \tau'  
        = 0,
    \\
     \label{eq:envar.mom}
    \begin{split}
    \int_{\R^d} \psi\cdot n \dd x\Big|_\sigma^\tau
    &-\int_\sigma^\tau\!\!\!\int_{\R^d}
    \psi_{\tau} \cdot n + \nabla \psi: \Big(- n\otimes \frac y2  + \frac{ n\otimes n}{\rho} + \ee^{d\tau'} \pp(\rho)\bbI \Big) -  \psi \cdot \ee^{d\tau'} \alpha n  \dd y\dd \tau'
    \\
    &\leq
    \int_\sigma^\tau \ee^{d\tau'}\mathcal K(\psi) \big( E-\EE(\tau;\rho,n\,\vert\,\rho_\infty,0)\big)\dd \tau',
    \end{split}
    \\
    \label{eq:envar.enin.0}
    \begin{split}
    \big[\ee^{\frac{d\tau'}{2}} E\big]\Big|_\sigma^\tau
       &+\int_\sigma^\tau\!\!\!\int_{\R^d} \ee^{\frac{d\tau'}{2}}\alpha \frac{|n|^2}{\rho} \dd y\dd \tau' 
       \\
        &\leq 
       \int_\sigma^\tau\!\!\!\int_{\R^d} -\ee^{\frac{d\tau'}{2}}
       \Big(h'(\rho_\infty)_\tau-\frac{y}{2}\cdot\nabla h'(\rho_\infty)\Big)(\rho-\rho_\infty)
       -\ee^{\frac{d\tau'}{2}}\nabla h'(\rho_\infty)\cdot n \dd y\dd \tau'
    \end{split}
    \end{align}
    for a.a.~$\sigma<\tau$.
    In virtue of~\cite[Lemma 2.5]{EitLas24envarhyp},
    we express the pointwise inequality~\eqref{eq:envar.enin.0} in a variational form as
    \[
    \begin{split}
    &\int_0^\infty \!\!- \varphi'\ee^{\frac{d\tau'}{2}} E 
    + \varphi\int_{\R^d} \ee^{\frac{d\tau'}{2}}\alpha \frac{|n|^2}{\rho} \dd y\dd \tau' 
       \\
       &
       \leq 
       \int_0^\infty\!\!\varphi\int_{\R^d} -\ee^{\frac{d\tau'}{2}}
       \Big(h'(\rho_\infty)_\tau-\frac{y}{2}\cdot\nabla h'(\rho_\infty)\Big)(\rho-\rho_\infty)
       -\ee^{\frac{d\tau'}{2}}\nabla h'(\rho_\infty)\cdot n \dd y\dd \tau'
       + \varphi(0) E(0)
    \end{split}
    \]
    for all $\varphi\in C_c^1([0,\infty))$ with $\varphi\geq 0$.
    We now use $\tau\mapsto\ee^{-d\tau/2}\varphi(\tau)$ as a test function and 
    employ~\cite[Lemma 2.5]{EitLas24envarhyp} again,
    to arrive at
    \begin{equation}
    \label{eq:envar.enin}
    \begin{split}
    E\big|_\sigma^\tau
       &+\int_\sigma^\tau\frac{d}{2} E\dd\tau'
       +\int_\sigma^\tau\!\!\!\int_{\R^d} \alpha \frac{|n|^2}{\rho} \dd y\dd \tau' 
       \\
       &\leq 
       \int_\sigma^\tau\!\!\!\int_{\R^d} 
       -\Big(h'(\rho_\infty)_\tau -\frac{y}{2}\cdot\nabla h'(\rho_\infty)\Big) (\rho-\rho_\infty) 
        -\nabla h'(\rho_\infty)\cdot n 
        \dd y\dd \tau'.
    \end{split}
    \end{equation}
    Summation of~\eqref{eq:envar.dens}, \eqref{eq:envar.mom} and~\eqref{eq:envar.enin} yields~\eqref{eq:envar}.
    Therefore, $(\rho,n,E)$ 
    is an energy-variational solution to~\eqref{eq:Euler.trans.md}.
\end{proof}

\subsection{Relative entropy inequality and long-time behavior}

We follow the proof of Theorem~\ref{thm:relentr.weak.E2}
to derive
a relative entropy inequality for energy-variational solutions.
Similarly to before, we define
    \begin{equation}
    \begin{aligned}
        \xi_1 &:=  - \nabla\Big( \frac{\ol n}{\ol\rho}\Big): \ee^{-d\tau} \rho \Big( \frac n\rho -\frac{\ol n}{\ol\rho} \Big)\otimes\Big( \frac n\rho -\frac{\ol n}{\ol\rho} \Big)
        - \dv\Big( \frac{\ol n}{\ol\rho}\Big) \pp (\rho \, \vert \, \ol \rho)  , \\
        \xi_2 &:=   \ee^{-d\tau }\frac{\rho}{\ol\rho} \Big(\frac n\rho -\frac{\ol n}{\ol\rho}\Big)
        \cdot \ol R, \\
        \xi_3 &:= -(\rho-\ol\rho) h''(\ol\rho) \ol R_1,
    \end{aligned}
    \label{eq:xi.def.md}
    \end{equation}
    where $\ol R:= \frac{\ol n}{\ol\rho} \cdot \ol R_1- \ol R_2$
    with
    \begin{equation}
\begin{aligned}
    \ol R_1
    &:=\ol\rho_\tau -\frac{y}{2}\cdot\nabla\ol\rho+\dv\ol n,
    \\
    \ol R_2
    &:=\ol n_\tau - \frac{y}{2}\cdot\nabla \ol n - \frac d2 \ol n + \dv\Big( \frac{\ol n\otimes\ol n}{\ol\rho} \Big) 
    + \ee^{d\tau} \big( \nabla\pp(\ol\rho) + \alpha \ol n\big),
\end{aligned}
\label{eq:residuals.md}
\end{equation}
and we set
\[
\damp(\rho,n\,\vert\,\ol\rho,\ol n) := \alpha \int_{\R^d} \rho \,\Big\vert \frac n\rho - \frac{\ol n}{\ol\rho} \Big\vert^2 \dd y.
\]
We deduce the following relative entropy inequality.

\begin{theorem}\label{thm:relentr.envar}
	Let $\hat\rho$ satisfy~\eqref{eq:rho.limit.angular}
	or $\hat\rho\equiv 0$.
    Let $(\rho,n,E)$ be an energy-variational solution
    to \eqref{eq:Euler.trans.md}
    with regularity weight $\mathcal K$, 
    reference density $\rho_\infty$,
    and initial data $(\rho_0,n_0)\in L^1_{\mathrm{loc}}(\R^d)\times L^1_{\mathrm{loc}}(\R^d;\R^d)$ satisfying 
    $\EE(0,\rho_0,n_0\,\vert\,\hat\rho,0)<\infty$.
    Let $\ol\rho\in C^1((0,\infty)\times\R^d)\cap C([0,\infty)\times\R^d)$ and
    $\ol n\in C^1((0,\infty)\times\R^d;\R^d)\cap C([0,\infty)\times\R^d;\R^d)$.
    We assume one of the following:
    \begin{enumerate}[label=(\alph*)]
        \item \label{item:relentr.envar.1}
        $\ol\rho-\hat\rho$ and $\ol n$ have compact support in space,
         or if $\hat\rho\equiv0$, then it holds $(\ol\rho,\ol n)=(0,0)$ and $\pp$ satisfies~\eqref{eq:pres.better}, 
        or
        \item \label{item:relentr.envar.2}
        $\pp$ is of the form~\eqref{eq:pres.gamma},
        and $\ol\rho$ satisfies
         $\inf\ol\rho>0$ and
    \[
    |\nabla\ol\rho|, \ |\nabla\ol n|, \ 
    \ol \rho_\tau,\ |\ol n_\tau|, \ 
    |y|(\ol\rho-\hat\rho), \ |y|\,|\ol n|
    \in L^\infty_{\mathrm{loc}}([0,\infty);L^1(\R^d)\cap L^\infty(\R^d)).
    \]
    \end{enumerate}
    Set
    \[
    \begin{aligned}
    \mathcal R(\tau)&:= E(\tau)-\EE(\tau;\rho(\tau),n(\tau)\,\vert\,\rho_\infty(\tau),0)+\EE(\tau;\rho(\tau),n(\tau) \, \vert \, \ol\rho(\tau),\ol n(\tau)),
    \\
    \mathcal R_0&:= E(0)-\EE(0;\rho_0,n_0\,\vert\,\rho_\infty(0),0)+\EE(0;\rho_0,n_0 \, \vert \, \ol\rho(0),\ol n(0)).
    \end{aligned}
    \]
    Then
    \begin{equation}
    \label{eq:relen.envar}
    \begin{aligned}
     &\int_0^\infty -\varphi'\, \mathcal R  
     +  \varphi \Big(\frac d2\mathcal R
     + \damp(\rho,n\,\vert\,\ol\rho,\ol n) \Big)\dd \tau
     \\
      &\qquad
      \leq \varphi(0)\mathcal R_0
      +\int_0^\infty \varphi \, \big[\Xi + \mathcal K(-\tfrac{\ol n}{\ol\rho})\big(E-\EE(\tau;\rho,n\,\vert\,\rho_\infty,0)\big)
      \big] \dd\tau
    \end{aligned}  
    \end{equation}
    for all $\varphi\in C^1_c([0,\infty))$ with $\varphi\geq 0$.
    Here $\Xi=\int_{\R^d} \xi_1+\xi_2+\xi_3\,\dd y$ with
    $\xi_j$, $j=1,2,3$, defined in~\eqref{eq:xi.def.md}.
\end{theorem}

\begin{proof}
For simplicity, consider a $\tau$-independent reference density 
$\rho_\infty\in C^1(\R^d)$ satisfying~\eqref{eq:rhoinfty.md},
which is no restriction, compare Remark~\ref{rem:envar.refdens}.
Using \cite[Lemma 2.5]{EitLas24envarhyp},
we express~\eqref{eq:envar} in a variational form as
    \[
    \begin{aligned}
       &\int_0^\infty -\varphi' \Big(E+\int_{\R^d} \phi\, \rho+\psi\cdot n \dd y\Big)
       + \varphi\Big(\frac{d}{2}E+\int_{\R^d} \alpha \frac{|n|^2}{\rho} \dd y\Big)\dd\tau
       \\
       &\qquad
       -\int_0^\infty\!\!\varphi\int_{\R^d} 
        \phi_{\tau} \,  \rho - \nabla\phi \cdot \big(\frac y2 \rho - n \big) - \frac d2 \phi \rho
        \dd y\dd \tau
       \\
       &\qquad\qquad
       -\int_0^\infty\!\!\varphi\int_{\R^d}
         \psi_{\tau} \cdot n + \nabla\psi: \Big(- n\otimes\frac y2 + \frac{ n\otimes n}{\rho} + \ee^{d\tau} \pp(\rho)\bbI \Big) -  \psi \cdot \ee^{d\tau} \alpha n  \dd y\dd \tau
        \\
        &\qquad\qquad\qquad
        \leq 
        \varphi(0)\Big(E(0)
        +\int_{\R^d} \phi(0,y)\rho_0(y) + \psi(0,y)\cdot n_0(y)\dd y\Big)
        \\
        &\qquad\qquad\qquad\qquad
        +\int_0^\infty\!\!\varphi\int_{\R^d} \nabla h'(\rho_\infty)\cdot \Big(\frac{y}{2}(\rho-\rho_\infty)-n\Big) \dd y\dd \tau
        \\
        &\qquad\qquad\qquad\qquad\qquad
        +\int_0^\infty \varphi\,\ee^{d\tau}\mathcal K(\psi) \big( E-\EE(\tau;\rho,n\,\vert\,\rho_\infty,0)\big)\dd \tau
    \end{aligned}
    \]
    for all $\phi,\psi \in C^1([0,\infty) \times \R)$
    and all $\varphi\in C_c^1([0,\infty))$ with $\varphi\geq 0$.
    Observe that $\phi$ and $\psi$ need not have 
    compact support in time due to the presence of $\varphi$.
    We consider
    $\phi(\tau,y)=\frac{1}{2}\ee^{-\tau} \big( \frac{\ol n (\tau,y) }{\ol\rho(\tau,y) } \big)^2 + h'(\rho_\infty(\tau,y))- h'(\ol\rho(\tau,y))$
    and 
    $\psi(\tau,y) = -\ee^{-\tau} \frac{\ol n(\tau,y)}{\ol\rho(\tau,y)}$.
    Then a direct calculation that uses the same identities 
    as in the proof of Lemma~\ref{le:relentr.E2.compact.supp}
    leads to~\eqref{eq:relen.envar}
    in case~\ref{item:relentr.envar.1}.
    In case~\ref{item:relentr.envar.2}, we can subsequently use the approximation argument 
    from the proof of Theorem~\ref{thm:relentr.weak.E2},
    using the coercivity estimate from Lemma~\ref{lem:coercivity},
    to extend the class of admissible functions $(\ol\rho,\ol n)$ as asserted.
\end{proof}

Notice that, analogously to Corollary~\ref{cor:ws.uniqueness},
one can immediately derive
a weak-strong uniqueness principle 
for energy-variational solutions
from Theorem~\ref{thm:relentr.envar}.

As in the one-dimensional case, 
we characterize the long-time behavior of 
the density
by a self-similar solution to 
the porous medium equation 
and associate the momentum by Darcy's law.
The corresponding similarity profile $(\rho^*,n^*)$ 
is determined by
\begin{equation} \label{eq:profileEq.md}
\begin{aligned}
0 = \frac1\alpha \Delta \pp(\rho^*) + \frac y2 \cdot \nabla \rho^* \quad \text{ and } \quad 0 = n^* + \frac1\alpha \nabla \pp(\rho^*)
\end{aligned}
\end{equation}
such that 
\[
\rho^*-\hat\rho\to 0,
\quad n^*\to 0
\qquad
\text{as } |y|\to\infty.
\]
While existence of such a similarity profile with good decay properties 
is known for $d=1$, compare Theorem~\ref{thm:profile},
a corresponding result seems to be unknown for higher dimension.
Therefore, we make the following assumption:
\begin{enumerate}[label=\textbf{(\Alph*)}]
    \item
    \label{item:A}
    It holds $\rho^*\in C^\infty(\R^d)$
    with $\rho^*(\R^d)\subseteq\hat\rho(\R^d)$ and
    \[
    |y|(\rho^*-\hat\rho), \ |y|\,|\nabla\rho^\ast|,\,|\nabla^2\rho^*|
    \in L^\infty_{\mathrm{loc}}([0,\infty);L^1(\R^d)\cap L^\infty(\R^d)).
    \]
\end{enumerate}
Assumption~\ref{item:A} is clearly satisfied if 
$\hat\rho$ is constant, 
which yields $\rho^*\equiv\hat\rho$.
Moreover, if $\pp$ is of the form~\eqref{eq:pres.gamma},
then one can
construct a smooth self-similar solution to the porous medium equation
$\wt\rho_t=\frac{1}{\alpha}\Delta\pp(\wt\rho)$,
which leads to a similarity profile $\rho^*$
satisfying~\eqref{eq:profileEq.md},
for instance,
by following the ideas of~\cite[Thm.~16.2]{Vazq07PMEM}.
However, the decay properties of $\rho^*$ are less clear,
and in multiple dimensions,
they may depend on the angular profile $\hat\rho$
and assumption~\ref{item:A} may be a proper restriction.

We can now proceed as in the case of weak solutions and conclude the convergence towards the similarity profile $(\rho^*,n^*)$ as $\tau \to\infty$.

\begin{theorem}
    \label{thm:conv.envar}
    Let one of the following assumptions be satisfied:
    \begin{enumerate}[label=(\alph*)]
        \item \label{item:conv.envar.1}
        Let $\alpha\geq0$, 
        let $\pp\in C([0,\infty))\cap C^1(0,\infty)$ with $\pp'>0$ in $(0,\infty)$,
        let $\rho^*=\hat\rho\geq0$ be constant and $n^*=0$,
        and if $\rho^*=0$, let $\pp$ satisfy~\eqref{eq:pres.better}, 
        or
        \item \label{item:conv.envar.2}
        let $\alpha>0$, 
        let $\pp$ satisfy~\eqref{eq:pres.gamma}, 
        let $\hat\rho$ satisfy~\eqref{eq:rho.limit.angular},
        and let $(\rho^*, n^*)$ satisfy~\eqref{eq:profileEq.md} and assumption~\ref{item:A}.
    \end{enumerate}
    Let $(\rho,n,E)$ be an energy-variational solution to~\eqref{eq:Euler.trans.md}
    subject to initial data $(\rho_0,n_0)\in L^1_{\mathrm{loc}}(\R^d)\times L^1_{\mathrm{loc}}(\R^d;\R^d)$
    such that $\EE(0;\rho_0,n_0 \, \vert \, \hat\rho, 0)<\infty$. 
    In case \ref{item:conv.envar.2}, let the corresponding regularity weight 
    $\mathcal K$ be given by~\eqref{eq:K}.
    Define 
    \[
    \begin{aligned}
    \mathcal R(\tau)&:= E(\tau)-\EE(\tau;\rho(\tau),n(\tau)\,\vert\,\rho_\infty,0)+\EE(\tau;\rho(\tau),n(\tau) \, \vert \, \rho^*,n^*),
    \\
    \mathcal R_0&:= E(0)-\EE(0;\rho_0,n_0\,\vert\,\rho_\infty,0)+\EE(0;\rho_0,n_0 \, \vert \, \rho^*, n^*).
    \end{aligned}
    \]
    Then
    \begin{equation} \label{eq:convergence.envar}
    \mathcal R(\tau) \leq \ee^{-(\frac d2{-}\theta)\tau + \frac d2 \mu} \Big( \mathcal R_0 + \frac{K}{\theta} \Big) \quad \text{ for a.a.~} \tau >0,
    \end{equation}
    where 
    $\theta:=\mu:=K/ \theta:=0$ in case~\ref{item:conv.envar.1},
    and     
    \[
    \begin{aligned}
    \theta &:= \frac{1}{\alpha}\max\big\{2\|[\nabla^2 h''(\rho^*)]_\mathrm{sym,+}\|_{L^\infty(\R^d)},(\gamma-1)\|[\Delta h''(\rho^*)]_+\|_{L^\infty(\R^d)}\big\}, 
    \\
    \mu &:= \Big\| \frac{|R^*|}{2 k (\rho^*)^\gamma} + \frac{3|R^*|}{2\rho^*}  \Big\|_{L^\infty(\R)},
    \qquad\quad
    K := \int_\R |R^*|\dd y,
    \\
    R^*&:=- \frac{y}{2}\cdot\nabla n^* - \frac12 n^* + \dv\Big( \frac{n^*\otimes n^*}{\rho^*} \Big)
    \end{aligned}
    \]
    in case~\ref{item:conv.envar.2}. 
    In particular, if $\theta<\frac{d}{2}$, then
    \begin{equation}
    \label{eq:convergence.envar.split}
    E(\tau)-\EE(\tau;\rho(\tau),n(\tau)\,\vert\,\rho_\infty,0)\to0
    \qquad
    \text{and}
    \qquad
    \EE(\tau;\rho(\tau),n(\tau) \, \vert \, \rho^*, n^*)
    \to 0
    \end{equation}
    as $\tau\to\infty$. 
\end{theorem}

\begin{proof}[Proof of Theorem~\ref{thm:conv.envar}]
	We use the relative entropy inequality~\eqref{eq:relen.envar} from Theorem~\ref{thm:relentr.envar} with the choice $(\ol\rho,\ol n)=(\rho^*,n^*)$
	which is admissible in both cases.
	    
    In case~\ref{item:conv.envar.1},
    we have $\xi_1=\xi_2=\xi_3=0$.
    Since $\mathcal K(0)=0$, 
    we obtain
    \[
     \int_0^\infty -\varphi'\, \mathcal R  
     +  \varphi \Big(\frac d2\mathcal R
     + \damp(\rho,n\,\vert\,\ol\rho,\ol n) \Big)\dd \tau
     \leq \varphi(0)\mathcal R_0,
    \]
    and the Gronwall inequality from Lemma~\ref{lemma:Gronwall.weak}
    yields~\eqref{eq:convergence.envar}.
    
    In case \ref{item:conv.envar.2}
    the terms $\xi_j$, $j=1,2,3$ from~\eqref{eq:xi.def.md} simplify,
    and a slight adaption of the proof of Lemma~\ref{le:estimates.Xi.profile}
    yields the estimate
    \[
    \Xi
    \leq (\theta+\mu\, \ee^{-d\tau/2})\, \EE (\tau;\rho,n \, \vert \, \rho^*,n^*) 
    +  \ee^{-d\tau/2}  K.
    \]
    Moreover, 
    using~\eqref{eq:h.properties},~\eqref{eq:profileEq.md} and~\eqref{eq:K}, we see that
    $
    \mathcal K(-\tfrac{n^*}{\rho^*})
    =\theta.
    $
    From~\eqref{eq:relen.envar} we now conclude the estimate
    \[
    \begin{aligned}
     &\int_0^\infty -\varphi'\, \mathcal R  
     +  \varphi \Big(\frac d2\mathcal R
     + \damp(\rho,n\,\vert\,\ol\rho,\ol n) \Big)\dd \tau
     \\
      &\qquad
      \leq \varphi(0)\mathcal R_0
      +\int_0^\infty \varphi \, \big[
      \mu\ee^{-d\tau/2 }\mathcal R 
      + K \ee^{-d\tau/2}
      + \theta \mathcal R\big)
      \big] \dd\tau,
    \end{aligned}  
    \]
    where we used $\EE(\tau;\rho,n\,\vert\,\rho^*,n^*) \leq \mathcal R$ 
    for the first term in the integral on the right-hand side.
    This inequality resembles~\eqref{eq:entr.ineq.final} in the proof of Theorem~\ref{thm:Conv}.
    Invoking now Gronwall's inequality from Lemma~\ref{lemma:Gronwall.weak}
    in the same way,
    we conclude~\eqref{eq:convergence.envar} in this case.

    Finally, if $\theta<\frac{d}{2}$,
    then~\eqref{eq:convergence.envar} implies $\mathcal R(\tau)\to0$ 
    as $\tau\to\infty$.
    Since $E\geq \EE(\tau;\rho,n\,\vert\,\rho_\infty,0)$,
    this shows~\eqref{eq:convergence.envar.split}.
\end{proof}

Theorem~\ref{thm:conv.envar} yields convergence of 
energy-variational solutions
towards the similarity profile
under the same flatness condition as for weak solutions, see Theorem~\ref{thm:Conv}. 
In a similar way as in Section~\ref{se:ConvProof}, 
we could also deduce convergence of the time-averaged kinetic energy.
In Theorem~\ref{thm:conv.envar}, we further deduce the convergence of the energy defect to zero,
which means that energy-variational solutions converge time-asymptotically to weak solutions in a certain sense.
Observe that in the case $\rho_+=\rho_-$ the result holds for any regularity weight $\mathcal K$,
while in the case $\rho_+\neq \rho_-$ we require the specific choice~\eqref{eq:K}. 
Clearly, one could also take a regularity weight 
that is larger than~\eqref{eq:K},
but this would lead to a more restrictive flatness condition.

\paragraph*{Acknowledgments.} 

T.~Eiter's research has been funded by Deutsche
Forschungsgemeinschaft (DFG) through grant CRC 1114 ``Scaling Cascades in
Complex Systems'', Project Number 235221301, Project YIP,
and through grant SPP 2410 ``Hyperbolic Balance Laws in Fluid Mechanics:~Complexity, Scales, Randomness (CoScaRa)'', Project Number
525941602.
The research of S.~Schindler has been funded by Deutsche Forschungsgemeinschaft (DFG) through
the Berlin Mathematics Research Center MATH+
(EXC-2046/1, DFG project no.\ 390685689) subproject ``DistFell''.
The authors thank Alexander Mielke for helpful discussions,
in particular when initiating the project.

\footnotesize

\addcontentsline{toc}{section}{References}

\newcommand{\etalchar}[1]{$^{#1}$}
\def\cprime{$'$}
\providecommand{\bysame}{\leavevmode\hbox to3em{\hrulefill}\thinspace}

\end{document}